\documentclass[11pt]{article}
\usepackage{amsmath,amsfonts,amssymb,amsthm}
\usepackage{eucal}
\usepackage{color}
\usepackage{makeidx}
\usepackage{multicol}
\usepackage{verbatim}

\def\N{{\mathbb N}}
\def\R{{\mathbb R}}
\def\P{{\mathbb P}}
\def\E{{\mathbb E}}
\def\I{{\mathbb I}}
\def\V{{\mathbb V\rm ar\,}}
\newcommand{\TR}{\widetilde{R}}
\newtheorem{theorem}{Theorem}
\newtheorem{lemma}[theorem]{Lemma} 
\newtheorem{corollary}[theorem]{Corollary}
\newtheorem{proposition}[theorem]{Proposition} 

\newtheorem{remark}[theorem]{Remark}
\makeindex

\title{Probabilistic Approach to Risk Processes\\ 
with Level-Dependent Premium Rate}

\author{Denis Denisov\footnote{University of Manchester, UK, 
denis.denisov@manchester.ac.uk},
Niklas Gotthardt\footnote{Augsburg University, Germany,
niklas.gotthardt@math.uni-augsburg.de},
Dmitry Korshunov\footnote{Lancaster University, UK, 
d.korshunov@lancaster.ac.uk},
and Vitali Wachtel\footnote{Bielefeld University, Germany, 
wachtel@math.uni-bielefeld.de}}

\begin{document}
\maketitle

\begin{abstract}
We study risk processes with level dependent premium rate. 
Assuming that the premium rate converges, as the risk reserve increases, 
to the critical value in the net-profit condition, 
we obtain upper and lower bounds for the ruin probability. 
In contrast to existing in the literature results, 
our approach is purely probabilistic and based on the analysis of Markov chains 
with asymptotically zero drift.

\vspace{5mm}


\end{abstract}

\section{Introduction}
\label{sec:intro}

\sectionmark{Risk processes}
\label{sec:risk}

In context of the collective theory of risk,
the classical {\it Cram\'er--Lundberg (Sparre Andersen) model}
is defined as follows. An insurance company receives
the constant inflow of premium at rate $v$, that is,
the premium income is assumed to be linear in time with rate $v$.
It is also assumed that the claims
incurred by the insurance company arrive according to
a homogeneous renewal process $N(t)$ with intensity $\lambda$
and the sizes (amounts) $\xi_n\ge 0$ of the claims
are independent copies of a random variable $\xi$ with finite mean.
The $\xi$'s are assumed independent of the process $N(t)$.
The company has an initial risk reserve $x=R(0)\ge0$.
Then the risk reserve $R(t)$ at time $t$ is equal to
\begin{eqnarray*}
R(t) &=& x+vt-\sum_{i=1}^{N(t)}\xi_i.
\end{eqnarray*}
The probability
\begin{eqnarray*}
\P\{R(t)\ge 0\mbox{ for all }t\ge0\}
&=& \P\Bigl\{\min_{t\ge 0}R(t)\ge 0\Bigr\}
\end{eqnarray*}
is the probability of ultimate survival and
\begin{eqnarray*}
\psi(x) &:=& \P\{R(t)<0\mbox{ for some }t\ge0\}\\
&=& \P\Bigl\{\min_{t\ge 0}R(t)<0\Bigr\}
\end{eqnarray*}
is the probability of ruin.\index{Ruin probability}
\index{Cram\'er--Lundberg!classical model!ruin probability} 
We have
\begin{eqnarray*}
\psi(x) &=& \P\Bigl\{\sum_{i=1}^{N(t)}\xi_i-vt>x
\mbox{ for some }t\ge0\Bigr\}.
\end{eqnarray*}
Since $v>0$, the ruin can only occur at a claim epoch. Therefore,
\begin{eqnarray*}
\psi(x) &=& \P\Bigl\{\sum_{i=1}^n\xi_i-vT_n>x
\mbox{ for some }n\ge1\Bigr\},
\end{eqnarray*}
where $T_n$ is the $n$th claim epoch, so that
$T_n=\tau_1+\ldots+\tau_n$ where the $\tau_k$'s are independent copies
of a random variable $\tau$ with finite mean $1/\lambda$, so that
$N(t):=\max\{n\ge 1:T_n\le t\}$.
Denote $X_i:=\xi_i-v\tau_i$ and $S_n:=X_1+\ldots+X_n$, then
\begin{eqnarray*}
\psi(x) &=& \P\Bigl\{\sup_{n\ge 1}S_n>x\Bigr\}.
\end{eqnarray*}
This relation represents the ruin probability problem
as the tail probability problem for
the maximum of the associated random walk $\{S_n\}$.
Let the {\it net-profit condition}
\begin{eqnarray}\label{net-profit}
v &>& v_c\ :=\ \E\xi/\E\tau\ =\ \lambda\E\xi
\end{eqnarray}
hold, thus $\{S_n\}$ has a negative drift: 
$\E S_1=\E\xi_1-v\E\tau<0$. 
Hence by the strong law of large numbers
$S_n\to-\infty$ a.s., so $\psi(x)\downarrow 0$ as $x\to\infty$.

The most classical case is when the distribution of $X_1$ 
satisfies the following well-known Cram\'er condition: there exists a $\beta>0$ such that
\begin{eqnarray}\label{def.beta}
\E e^{\beta X_1} &=& 1.
\end{eqnarray}
Under this condition, the sequence $e^{\beta S_n}$ is a martingale and, 
by the Doob maximal inequality, the following Lundberg's inequality
\begin{equation}
\label{eq:Cramer_bound}
\psi(x)=\P\left\{\sup_{n\ge1}e^{\beta S_n}>e^{\beta x}\right\}
\le e^{-\beta x},\quad x>0.
\end{equation}
If we additionally assume that 
$$
\E X_1e^{\beta X_1}<\infty
$$
and the distribution of $X_1$ is non-lattice,
then the Cram\'er--Lundberg approximation holds, that is, 
there exists a constant $c_0\in(0,1)$ such that 
\begin{equation}
\label{eq:Cramer_asymp}
\psi(x)\sim c_0e^{-\beta x}\quad\text{as }x\to\infty,
\end{equation}
see e.g. Theorem VI.3.2 in Asmussen and Albrecher \cite{AA};
in the lattice case $x$ must be taken as a multiple of the lattice step.
The most important feature of these results is the fact that the upper bound 
\eqref{eq:Cramer_bound} and the asymptotic relation 
\eqref{eq:Cramer_asymp} depend on the distribution of $X_1$ only via the parameter $\beta$. 
If the moment condition \eqref{def.beta} on the distribution of $X_1$ does not hold 
then the tail asymptotics for $\psi(x)$ are typically determined by the tail of the claim size $\xi$. 
The most prominent situation is when the distribution of $\xi$ is of subexponential type. 
We discuss this case in more detail later.

The risk models with non-constant premium rates have also 
become rather popular in the collective risk literature.
There are two main approaches, one of them leads to
a Markovian model when 
the premium rate is a function of the current level of the risk reserve 
$R(t)$, see e.g. Asmussen and Albrecher \cite[Chapter VIII]{AA},
Albrecher et al. \cite{A_all}, Boxma and Mandjes \cite{BM}, 
Czarna et al. \cite{CPRY},
Marciniak and Palmowski et al. \cite{MZ}. 
The second approach considers the premium rate that depends on the whole claims history,
see e.g. Li et al. \cite{LNC}.

In this paper we follow the first approach and consider a risk process where the premium
rate $v(y)$ only depends on the current level of risk reserve $R(t)=y$,
so $R(t)$ satisfies the equality
\begin{eqnarray}\label{risk.1}
R(t) &=& x+\int_0^t v(R(s))ds-\sum_{j=1}^{N(t)}\xi_j;
\end{eqnarray}
$v(y)$ is assumed to be a measurable non-negative bounded function.
The probability of ruin given initial risk reserve $x$
is again denoted by $\psi(x)$. 
Notice that $\psi(x)$ is no longer a decreasing function of $x$
as it is in the classical case.

The ruin probability for such processes with level dependent premium rate 
is much less studied in the literature than with constant premium rate, 
and all known results are exact expressions for some particular distributions of $\tau$, 
$\xi$ and/or for particular choices of the rate function $v(y)$.
The first example of the risk process where $\psi(x)$
is explicitly calculable is the case of exponentially distributed
$\tau$ and $\xi$, say with parameters $\lambda$ and $\mu$ respectively,
so hence $v_c=\lambda/\mu$.
In this case, for some $c_0\in(0,1)$,
\begin{eqnarray}\label{psi.x.formula}
\psi(x) &=& c_0\int_x^\infty \frac{1}{v(y)}
\exp\Bigl\{-\mu y+\lambda\int_0^y\frac{dz}{v(z)}\Bigr\}dy\nonumber\\
&=& c_0\int_x^\infty \frac{1}{v(y)}
\exp\Bigl\{\lambda\int_0^y\Bigl(\frac{1}{v(z)}-\frac{1}{v_c}\Bigr)dz\Bigr\}dy,
\end{eqnarray}
provided the outer integral is convergent from $0$ to infinity,
see, e.g. Corollary 1.9 in Albrecher and Asmussen \cite[Ch. VIII]{AA}. 
Some further examples of solutions in closed form can be found in Albrecher et al \cite{A_all}. 
The authors of that paper use a purely analytical approach, 
which works however only in the situations where the Laplace transforms of $\xi$ 
and $\tau$ are rational functions.

The main goal of our paper is to develop a probabilistic method of the asymptotic analysis 
of risk processes with level-dependent premium rate, 
which is not based on exact calculations and uses only moment 
and tail conditions on $\xi$ and $\tau$. 
The following two qualitatively different cases can be identified:
\begin{eqnarray}\label{vyvinfty}
v(y) &\to& v_\infty\ >\ v_c\quad\mbox{as }y\to\infty;\\
v(y) &\to& v_c\quad\mbox{as }y\to\infty.\label{risk.appr}
\end{eqnarray}
In the first case \eqref{vyvinfty} one could expect that the ruin probability $\psi(x)$
decays similarly to the classical collective risk model with constant premium rate $v_\infty$.
In this paper we concentrate on a more difficult case \eqref{risk.appr}
where the ruin is more likely due to the approaching the critical
premium rate. 

We start by discretising the time; this procedure is standard for risk processes
with constant rate.
Since the ruin can only occur at a claim epoch,
the ruin probability may be reduced to that for
the embedded Markov chain $R_n:=R(T_n)$, $n\ge 1$, $R_0:=x$, that is,
\begin{eqnarray*}
\psi(x) &=& \P\{R_n<0\mbox{ for some }n\ge0\}.
\end{eqnarray*}
So, our main goal is to analyse the down-crossing probabilities for the chain $\{R_n\}$.
In contrast to the constant premium case, we deal with a Markov chain instead 
of a random walk with independent increments.

As mentioned above, we shall restrict our attention to the case \eqref{risk.appr} 
where $v(y)$ approaches the critical value $v_c$ at infinity.
Then the Markov chain $\{R_n\}$ has asymptotically zero drift, that is,
\begin{eqnarray}\label{Mcwazd}
\E\{R_1-R_0\mid R_0=x\} &\to& 0\quad\mbox{as }x\to\infty,
\end{eqnarray}
see Theorem \ref{non.exp.return} below. The study of Markov chains
with vanishing drift was initiated by Lamperti in a series of papers
\cite{Lamp60, Lamp62, Lamp63}. For further development in Lamperti's problem,
see Menshikov et al \cite{AMI, MPW16}.
We also show in Theorem \ref{non.exp.return} that under \eqref{Mcwazd} 
the ruin probability decays slower than any exponential function, that is, for any $\lambda>0$,
\begin{eqnarray*}
e^{\lambda x}\psi(x) &\to& \infty\quad\mbox{as }x\to\infty.
\end{eqnarray*}

As well-known, for the classical Cram\'er--Lundberg model under 
the net-profit condition \eqref{net-profit}, 
the ruin probability decays slower than any exponential function 
if and only if the claim size tail distribution is so.
As just mentioned, risk processes under the condition \eqref{risk.appr} 
give rise to heavy-tailed ruin probabilities whatever the distribution of the claim size, 
even if it is a bounded random variable. 
So, risk processes with near critical premium rate provide an important example 
of a stochastic model where light-tailed input produces heavy-tailed output.

We want to investigate how the rate of convergence 
in \eqref{risk.appr} is reflected in how slowly the ruin probability $\psi(x)$ decreases.
Let us get some intuition on what kind of phenomena we could expect here by considering
the case of exponentially distributed $\xi$ and $\tau$. As we have mentioned above, 
the ruin probability $\psi(x)$ is given in this case by \eqref{psi.x.formula}.
Combining \eqref{psi.x.formula} and \eqref{risk.appr}, we obtain
\begin{eqnarray*}
\psi(x) &\sim& \frac{c_0}{v_c}\int_x^\infty
\exp\Bigl\{\lambda\int_0^y\Bigl(\frac{1}{v(z)}-\frac{1}{v_c}\Bigr)dz\Bigr\}dy
\quad\mbox{as }x\to\infty.
\end{eqnarray*}

If the premium rate $v(z)\ge v_c$ approaches $v_c$ at the rate of $\theta/z$,
$\theta>0$, more precisely, if
\begin{eqnarray}\label{risk.2}
\Bigl|v(z)-v_c-\frac{\theta}{z}\Bigr| &\le& p(z)\quad\mbox{for all }z>1,
\end{eqnarray}
where $p(z)>0$ is an integrable at infinity decreasing
function, then we get
\begin{eqnarray*}
\frac{1}{v(z)} &=&
\frac{1}{v_c}-\frac{\theta}{v_c^2z}+O(p(z)+1/z^2)
\end{eqnarray*}
and consequently
\begin{eqnarray*}
\lambda\int_0^y\Bigl(\frac{1}{v(z)}-\frac{1}{v_c}\Bigr)dz
&=& -\frac{\theta\mu^2}{\lambda}\log y +c_1+o(1)\quad\mbox{as }y\to\infty,
\end{eqnarray*}
where $c_1$ is a finite real. Let $\theta>\lambda/\mu^2$.
Then, for $C:=c_0e^{c_1}/(\theta\mu-\lambda/\mu)>0$,
\begin{eqnarray}\label{risk.exp}
\psi(x) &\sim& \frac{C}{x^{\theta\mu^2/\lambda-1}}\quad\mbox{as }x\to\infty.
\end{eqnarray}
A similar asymptotic equivalence can be obtained also in the case where the
Laplace transforms of variables $\xi_1$ and $\tau_1$ are rational functions,
see Albrecher et al. \cite{A_all}.

If the premium rate $v(z)$ approaches $v_c$ at a slower rate of $\theta/z^\alpha$,
$\theta>0$ and $\alpha\in(0,1)$, more precisely, if
\begin{eqnarray}\label{risk.2.hx}
\Bigl|v(z)-v_c-\frac{\theta}{z^\alpha}\Bigr| &\le& p(z)\quad\mbox{for all }z>1,
\end{eqnarray}
where $p(z)>0$ is an integrable at infinity decreasing function, then we get
\begin{eqnarray*}
\frac{1}{v(z)} &=&
\frac{1}{v_c}\sum_{j=0}^\infty \Bigl(-\frac{\theta}{v_c}\Bigr)^j
\frac{1}{z^{\alpha j}}+O(p(z)).
\end{eqnarray*}
Let $\gamma:=\min\{k\in\N:k\alpha>1\}$. Then
\begin{eqnarray*}
\frac{1}{v(z)} &=&
\frac{1}{v_c}\sum_{j=0}^{\gamma-1} \Bigl(-\frac{\theta}{v_c}\Bigr)^j
\frac{1}{z^{\alpha j}}+O(p_1(z)),
\end{eqnarray*}
where $p_1(z)=p(z)+z^{-\gamma\alpha}$ is integrable at infinity.
Consequently, if $1/\alpha$ is not integer, then
\begin{eqnarray*}
\lambda\int_0^y\Bigl(\frac{1}{v(z)}-\frac{1}{v_c}\Bigr)dz
&=& \frac{\lambda}{v_c} \int_1^y
\sum_{j=1}^{\gamma-1} \Bigl(-\frac{\theta}{v_c}\Bigr)^j
\frac{1}{z^{\alpha j}}dz+c_2+o(1)\\
&=& \frac{\lambda}{v_c}
\sum_{j=1}^{\gamma-1} \Bigl(-\frac{\theta}{v_c}\Bigr)^j
\frac{y^{1-\alpha j}}{1-\alpha j}+c_3+o(1)\quad\mbox{as }y\to\infty,
\end{eqnarray*}
where $c_3$ is a finite real because $p_1(x)$ is integrable.
In the case of integer $1/\alpha$ one has
\begin{eqnarray*}
\lambda\int_0^y\Bigl(\frac{1}{v(z)}-\frac{1}{v_c}\Bigr)dz
&=& \frac{\lambda}{v_c}
\sum_{j=1}^{\gamma-2} \Bigl(-\frac{\theta}{v_c}\Bigr)^j
\frac{y^{1-\alpha j}}{1-\alpha j}\\
&&\hspace{20mm}+
\frac{\lambda}{v_c} \Bigl(-\frac{\theta}{v_c}\Bigr)^{\gamma-1}\log y
+c_4+o(1)\quad\mbox{as }y\to\infty.
\end{eqnarray*}

Let, for example, $\alpha\in(1/2,1)$. Then
\begin{eqnarray*}
\lambda\int_0^y\Bigl(\frac{1}{v(z)}-\frac{1}{v_c}\Bigr)dz
&=& -\frac{\theta\mu^2}{\lambda(1-\alpha)} y^{1-\alpha} +c_3+o(1)\quad\mbox{as }y\to\infty.
\end{eqnarray*}
Therefore, for $C_1:=c_0e^{c_3}/\theta\mu>0$ and $C_2:=\theta\mu^2/\lambda(1-\alpha)>0$,
\begin{eqnarray}\label{risk.exp.new}
\psi(x) &\sim& C_1x^\alpha e^{-C_2x^{1-\alpha}}\quad\mbox{as }x\to\infty.
\end{eqnarray}

We are going to extend these results to not necessarily exponential distributions 
where there are no closed form expressions like \eqref{psi.x.formula} available for $\psi(x)$. 
In that case we can only derive lower and upper bounds for $\psi(x)$
which have the same decay rate at infinity.

\section{Heavy-tailedness of the ruin probability in the critical case}
\label{sec:heavy}

Denote the jumps of the chain $\{R_n=R(T_n)\}$ by $\xi(x)$,
that is,
$$
\P\{\xi(x)\in B\}=\P\{R_1-R_0\in B\mid R_0=x\}
$$
for all Borel sets $B$.
The dynamics of the risk reserve between two consequent claims
is governed by the differential equation $R'(t)=v(R(t))$.
Let $V_x(t)$ denote its solution with initial value $x$, so then
\begin{eqnarray*}
V_x(t) &=& x+\int_0^t v(V_x(s))ds.
\end{eqnarray*}
Therefore,
\begin{eqnarray*}
\xi(x) &=_{\rm st}& V_x(\tau)-x-\xi\ =\ \int_0^\tau v(V_x(s))ds-\xi,
\end{eqnarray*}
where $=_{\rm st}$ stands for the equality in distribution.

To avoid trivial case where $\psi(x)=0$ for all sufficiently large $x$,
we assume that $\psi(x)>0$ for all $x$.
As the function $\psi(x)$ is not necessarily decreasing, 
we need in the sequel a stronger condition: for all $x_0$,
\begin{eqnarray}\label{cond:Rgr0}
\inf_{x\le x_0} \psi(x) &>& 0.
\end{eqnarray}
A sufficient condition for that is that, for all $x_0$, 
there exists an $\varepsilon=\varepsilon(x_0)>0$ such that
\begin{eqnarray}\label{cond:Rgr01}
\P\{\xi(x)\le-\varepsilon\} &\ge& \varepsilon\quad\mbox{for all }x\le x_0.
\end{eqnarray}
In its turn, it is sufficient to assume that the random variable 
$\xi$ is unbounded, due to the inequality $\xi(y)\le \overline v\tau-\xi$ valid for all $y$,
where $\overline v:=\sup_{z>0} v(z)$.

\begin{theorem}\label{non.exp.return}
Let $v_c=\E\xi/\E\tau$ and let $v(x)\to v_c$ as $x\to\infty$. 
Then the chain $\{R_n\}$ has asymptotically zero drift, that is, \eqref{Mcwazd} holds true.

If, in addition, \eqref{cond:Rgr0} holds true, then,
for all $\lambda>0$, $e^{\lambda x}\psi(x)\to\infty$ as $x\to\infty$.
\end{theorem}

\begin{proof}
Since $v(y)\to v_c$, for all $t>0$,
\begin{eqnarray*}
\int_0^t v(V_x(s))ds &\to& v_ct\quad\mbox{as }x\to\infty.
\end{eqnarray*}
This implies the following convergence in distribution:
\begin{eqnarray*}
\xi(x) &\Rightarrow& v_c\tau-\xi\quad\mbox{as }x\to\infty,
\end{eqnarray*}
which implies the first statement. It also implies that, for all $\lambda>0$,
\begin{eqnarray*}
(e^{-\lambda\xi(x)}-1)\I\{\xi(x)>-x\} &\Rightarrow& 
e^{\lambda(\xi-v_c\tau)}-1\quad\mbox{as }x\to\infty.
\end{eqnarray*}
Hence, as follows from Fatou's Lemma, 
\begin{eqnarray*}
\liminf_{x\to\infty}\ 
\E(e^{-\lambda\xi(x)}-1)\I\{\xi(x)>-x\} &\ge& \E e^{\lambda(\xi-v_c\tau)}-1\\
&>& e^{\lambda\E(\xi-v_c\tau)}-1.
\end{eqnarray*}
Recalling that $v_c=\E\xi/\E\tau$, we get $\E(\xi-v_c\tau)=0$.
Therefore, for all $\lambda>0$
there exists an $\varepsilon=\varepsilon(\lambda)>0$ such that 
\begin{eqnarray}\label{Fatou}
\E(e^{-\lambda\xi(x)}-1)\I\{\xi(x)>-x\} 
&\ge& \varepsilon\quad\mbox{for all sufficiently large }x.
\end{eqnarray}

Let $\lambda>0$. Consider a bounded decreasing function 
$U_\lambda(x):=\min(e^{-\lambda x},1)$. For all $x>0$,
\begin{eqnarray*}
\E(U_\lambda(x+\xi(x))-U_\lambda(x)) &\ge& 
\E\{e^{-\lambda(x+\xi(x))}-e^{-\lambda x};\ x+\xi(x)>0\}\\
&=& e^{-\lambda x}\E\{e^{-\lambda \xi(x)}-1;\ \xi(x)>-x\}.
\end{eqnarray*}
Due to \eqref{Fatou}, there exists a sufficiently large $x_\lambda>0$ such that 
\begin{eqnarray*}
\E(U_\lambda(x+\xi(x))-U_\lambda(x)) &\ge& 0\quad\mbox{for all }x>x_\lambda.
\end{eqnarray*}
Therefore, the process $\{U_\lambda(R_{n\wedge\tau_{B_\lambda}})\}$ 
is a bounded submartingale, where $B_\lambda:=(-\infty,x_\lambda]$
and $\tau_B=\min\{n:R_n\in B\}$.
Hence by the optional stopping theorem, for $z>x_\lambda$ 
and $x\in(x_\lambda,z)$,
$$
\E_x U_\lambda(R_{\tau_{B_\lambda}\wedge\tau_{(z,\infty)}})
\ \ge\ \E_x U_\lambda(X_0)\ =\ U_\lambda(x).
$$
Letting $z\to\infty$ we conclude that
\begin{eqnarray*}
\E_x\{U_\lambda(R_{\tau_{B_\lambda}});\ \tau_{B_\lambda}<\infty\} &=& 
\lim_{z\to\infty}\E_x\{U_\lambda(R_{\tau_{B_\lambda}});\ \tau_{B_\lambda}<\tau_{(z,\infty)}\}\\
&=& \lim_{z\to\infty}\E_x U_\lambda(R_{\tau_{B_\lambda}\wedge\tau_{(z,\infty)}})
-\lim_{z\to\infty}\E_x\{U_\lambda(R_{\tau_{(z,\infty)}});\ \tau_{B_\lambda}>\tau_{(z,\infty)}\}\\
&\ge& U_\lambda(x)-0\ =\ U_\lambda(x).
\end{eqnarray*}
On the other hand, since $U_\lambda$ is bounded by $1$,
\begin{eqnarray*}
\E_x\{U_\lambda(R_{\tau_{B_\lambda}});\ \tau_{B_\lambda}<\infty\} 
&\le& \P_x\{\tau_{B_\lambda}<\infty\}.
\end{eqnarray*}
This allows us to deduce the lower bound
\begin{eqnarray*}
\P_x\{\tau_{B_\lambda}<\infty\} &\ge& U_\lambda(x)\ =\ e^{-\lambda x}
\quad\mbox{for all }x>x_\lambda,
\end{eqnarray*}
and hence the conclusion (ii) follows, because by the Markov property, 
for all $\lambda>0$ and $x>0$,
\begin{eqnarray}
\label{lower.hat.0}
\psi(x)\ =\ \P_x\{\tau_{(-\infty,0]}<\infty\} &\ge& 
\P_x\{\tau_{B_\lambda}<\infty\}
\inf_{y\in(0,x_\lambda]} \P_y\{\tau_{(-\infty,0)}<\infty\}\nonumber\\
&\ge& \delta(\lambda)\P_x\{\tau_{B_\lambda}<\infty\},
\end{eqnarray}
where $\delta(\lambda)=(\varepsilon(x_\lambda))^{x_\lambda/\varepsilon(x_\lambda)}$,
owing to the condition \eqref{cond:Rgr0}.
\end{proof}

\section{Transience of the underlying Markov chain}
\label{sec:transience}
In this section we find conditions on the rate function $v(z)$
which ensure that the ruin probability is strictly less than one.

\begin{theorem}\label{thm:transience}
Let, for some $\theta>0$, 
\begin{eqnarray}\label{upper.for.vz}
v(z) &\ge& v_c+\theta/z\quad\mbox{for all sufficiently large }z.
\end{eqnarray}
Let both $\E \tau_1^2$ and $\E\xi_1^2$ be finite. If
\begin{eqnarray}\label{eq:theta-cond}
\theta &>& \frac{\V\xi+v_c^2\V\tau}{2\E\tau},
\end{eqnarray}
then the underlying Markov chain $\{R_n=R(T_n)\}$ is transient or, 
equivalently, $\psi(x)<1$ for all $x>0$. 

If, in addition,
\begin{eqnarray}\label{asy.for.vz}
v(z) -v_c &\sim& \theta/z\quad\mbox{as }z\to\infty,
\end{eqnarray}
then $R_n^2/n$ weakly converges to a $\Gamma$-distribution with
mean $2\mu+b$ and variance $(2\mu+b)2b$ where
$\mu:=\theta\E\tau$ and $b:=\V\xi+v_c^2\V\tau$.
\end{theorem}

As we see from the convergence to a $\Gamma$-distribution,
in the case \eqref{eq:theta-cond} the chain $R_n$ escapes to infinity in probability
at rate $\sqrt n$ in quite specific way as there is now law of large numbers.
In the case where $v(z)-v_c\sim c/z^\alpha$ with $\alpha\in(0,1)$,
the chain $R_n$ is transient too, however as follows from Lamperti 
\cite[Theorem 7.1]{Lamp62}, it follows a law of large numbers,
$R_n^{1+\alpha}/n\to c(1+\alpha)$ as $n\to\infty$.

Below we prove Theorem \ref{thm:transience} via Lyapunov (test) functions approach, so 
we start with moment computations for the jumps of  $\{R_n\}$.
Denote by $m_k(x)$ the $k$th moment of the jump $\xi(x)$ 
of the chain $\{R_n\}$ from state $x$, that is,
$m_k(x) =\E \xi^k(x)$.

\begin{lemma}\label{l:risk}
If both $\E \tau^2$ and $\E\xi^2$ are finite, then, 
under the rate of convergence \eqref{risk.2}, as $x\to\infty$,
\begin{eqnarray}\label{risk.7}
m_1(x) &=& \frac{\theta \E\tau}{x}+O(p(x)+1/x^2),\\
\label{risk.8}
m_2(x) &=& \V\xi+v_c^2\V\tau+O(1/x).
\end{eqnarray}
Let $\delta>0$. If $\E\xi^{\gamma_0+2}<\infty$ for some $\gamma_0\ge 0$ then
\begin{eqnarray}\label{cond.xi.le.return}
\P\{\xi(x)<-\delta x\} &=& o(p_1(x)/x^{\gamma_0+1})\quad\mbox{as }x\to\infty,
\end{eqnarray}
for some decreasing integrable at infinity function $p_1(x)$.

If $\E \tau^2\log(1+\tau)<\infty$ and $\E\xi^2\log(1+\xi)<\infty$, 
then, as $x\to\infty$,
\begin{eqnarray}\label{cond.3.moment.return}
\E\bigl\{|\xi(x)|^3;\ |\xi(x)|\le \delta x\bigr\} &=& o(x^2p_1(x)),\\
\label{risk.8.tail}
\E\{\xi^2(x);\ |\xi(x)|>\delta x\} &=& o(xp_1(x)).
\end{eqnarray}
\end{lemma}

\begin{proof}
By \eqref{risk.2},
\begin{eqnarray*}
v(y) &\le& v_c+\theta/y+p(y)\\
&\le& v_c+\theta/x+p(x)\quad\mbox{for all }y\ge x,
\end{eqnarray*}
therefore
\begin{eqnarray}\label{risk.Vx.above}
\nonumber
V_x(t)-x &=&\int_0^t v(V_x(s))ds\\
&\le& v_ct+\theta t/x+p(x)t,\quad t>0.
\end{eqnarray}
On the other hand, again by \eqref{risk.2},
\begin{eqnarray*}
v(y) &\ge& v_c+\theta/y-p(y)\\
&\ge& v_c+\theta/y-p(x)\quad\mbox{for all }y\ge x,
\end{eqnarray*}
Hence,
\begin{eqnarray*}
V_x(t)-x &\ge& v_ct+\theta\int_0^t \frac{ds}{V_x(s)}-p(x)t\\
&\ge& v_ct+\theta\int_0^t \frac{ds}{x+(v_c+\theta/x+p(x))s}-p(x)t\\
&=& v_ct+\frac{\theta}{v_c+\theta/x+p(x)}
\log\bigl(1+(v_c+\theta/x+p(x))t/x\bigr)-p(x)t,
\end{eqnarray*}
where the second inequality follows from the upper bound \eqref{risk.Vx.above}. Therefore,
\begin{eqnarray}\label{risk.Vx.below}
V_x(t)-x &\ge& v_ct+\frac{\theta}{v_c+\theta/x+p(x)}
\log\bigl(1+v_ct/x\bigr)-p(x)t.
\end{eqnarray}
Since $\xi(x)=V_x(\tau)-x-\xi$,
it follows from \eqref{risk.Vx.above} and \eqref{risk.Vx.below} that
\begin{eqnarray}\label{risk.6}
\lefteqn{v_c\tau-\xi+\frac{\theta}{v_c+\theta/x+p(x)}
\log\Bigl(1+\frac{v_c\tau}{x}\Bigr)-p(x)\tau}\nonumber\\
&&\hspace{40mm}\le\ \xi(x)\ \le\
v_c\tau-\xi +\frac{\theta\tau}{x}+p(x)\tau.\hspace{10mm}
\end{eqnarray}
Recalling that $v_c=\E\xi/\E\tau$, we get
\begin{eqnarray*}
\frac{\theta}{v_c+\theta/x+p(x)}\E\log\Bigl(1+\frac{v_c\tau}{x}\Bigr)-p(x)\E\tau
&\le& m_1(x)\ \le\ \frac{\theta}{x}\E\tau+p(x)\E\tau.
\end{eqnarray*}
By the inequality $\log(1+z)\ge z-z^2/2$ for $z\ge 0$,
\begin{eqnarray*}
\E\log\Bigl(1+\frac{v_c\tau}{x}\Bigr) &\ge&
\frac{v_c \E\tau}{x}-\frac{v^2_c \E\tau^2}{2x^2}.
\end{eqnarray*}
Therefore, the relation \eqref{risk.7} follows.
From that expression we have
\begin{eqnarray*}
m_2(x) &=& \V\xi(x)+m_1^2(x)\\
&=& \V(V_x(\tau)-x-\xi)+O(p^2(x)+1/x^2)\\
&=& \V(V_x(\tau)-x)+\V\xi+O(p^2(x)+1/x^2)
\quad\mbox{as }x\to\infty.
\end{eqnarray*}
Recalling that
\begin{eqnarray*}
v_ct-p(x)t &\le& V_x(t)-x\ \le\ v_ct+\frac{\theta}{x}t+p(x)t,
\end{eqnarray*}
we get
\begin{eqnarray*}
(v_c-p(x))\E\tau &\le& \E (V_x(\tau)-x)\ \le\ (v_c+\theta/x+p(x))\E\tau
\end{eqnarray*}
and
\begin{eqnarray*}
(v_c-p(x))^2\E \tau^2 &\le& \E (V_x(\tau)-x)^2
\ \le\ (v_c+\theta/x+p(x))^2\E\tau^2.
\end{eqnarray*}
Hence,
\begin{eqnarray*}
\V(V_x(\tau)-x) &=& v_c^2\V\tau+O(1/x)
\quad\mbox{as }x\to\infty,
\end{eqnarray*}
which in its turn implies \eqref{risk.8}.

Next, since $V_x(\tau)-x\ge 0$ and $\xi\ge 0$, we have
\begin{eqnarray}\label{risk.8a.ge}
\xi^2(x)\I\{\xi(x)>\delta x\} &=& (V_x(\tau)-x-\xi)^2\I\{V_x(\tau)-x-\xi>\delta x\}\nonumber\\
&\le& (V_x(\tau)-x)^2\I\{V_x(\tau)-x>\delta x\}\nonumber\\
&\le& \overline v^2\tau^2\I\{\tau>\delta x/\overline v\},
\end{eqnarray}
where $\overline v=\sup_{z} v(z)$, owing to the inequality $V_x(t)-x\le \overline vt$,
which follows from \eqref{risk.Vx.above}. Similarly,
\begin{eqnarray}\label{risk.8a.le}
\xi^2(x)\I\{\xi(x)<-\delta x\} &=& (V_x(\tau)-x-\xi)^2\I\{V_x(\tau)-x-\xi<-\delta x\}\nonumber\\
&\le& \xi^2\I\{\xi>\delta x\}.
\end{eqnarray}
Then it follows from the finiteness of $\E\xi^2\log(1+\xi)$ and
$\E\tau^2\log(1+\tau)$ that both tail expectations
$\E\{\tau^2;\ \tau>\delta x/\overline v\}$ and
$\E\{\xi^2;\ \xi>\delta x\}$ are of order $o(xp_1(x))$ 
for some decreasing integrable at infinity function $p_1(x)$,
see Lemma \ref{l:maj.p.e.log}.
Hence the upper bound \eqref{risk.8.tail}.

Further, the upper bound \eqref{cond.3.moment.return} 
follows from Lemma \ref{l:p.V.maj.o} with $\gamma=2$ and $\alpha=1$.

Finally, 
\begin{eqnarray*}
\P\{\xi(x)<-\delta x\} &\le& \P\{\xi>\delta x\}\ =\ o(p_1(x)/x^{\gamma_0+1}),
\end{eqnarray*}
for some decreasing integrable at infinity function $p_1(x)$,
due to Lemma \ref{l:maj.p.e} with $\gamma=\gamma_0+2$, $\beta=0$, and $\alpha=1$.
Hence the upper bound \eqref{cond.xi.le.return}.
\end{proof}

\begin{proof}[Proof of Theorem \ref{thm:transience}.]
Let us consider the function $v_\theta(z):=\min(v(z),v_c+\theta/z)$.
The dynamics of the risk reserve between two consequent claims
with premium rate $v_\theta(z)$ is governed by the differential equation $R'(t)=v_\theta(R(t))$.
Let $V_{\theta,x}(t)$ denote its solution with the initial value $x$, so then
\begin{eqnarray*}
V_{\theta,x}(t) &=& x+\int_0^t v_\theta(V_{\theta,x}(s))ds.
\end{eqnarray*}
Since $v_\theta(z)\le v(z)$, 
\begin{eqnarray}\label{V.compar}
V_x(t) &\ge& V_{\theta,x}(t)\quad\mbox{for all }t>0.
\end{eqnarray}

For $\xi_\theta(x):=V_{\theta,x}(\tau)-x-\xi$, denote $m_{\theta,k}(x):=\E\xi_\theta^k(x)$. Since 
$v_\theta(z)=v_c+\theta/z$ for all sufficiently large $z$,
Lemma ~\ref{l:risk} applies. As a result we have
\begin{eqnarray*}
m_{\theta,1}(x) &=& \frac{\theta \E\tau}{x}+O(1/x^2)\quad\mbox{as }x\to\infty,
\end{eqnarray*}
and
\begin{eqnarray*}
m_{\theta,2}(x) &=& \V\xi+v_c^2\V\tau+O(1/x)
\quad\mbox{as }x\to\infty.
\end{eqnarray*}
Therefore,
\begin{eqnarray*}
\frac{2m_{\theta,1}(x)}{m_{\theta,2}(x)} &=&
\frac{2\theta\E \tau}{\V\xi+v_c^2\V\tau} \cdot
\frac{1}{x}+O(1/x^2)\quad\mbox{as }x\to\infty.
\end{eqnarray*}
By the condition on $\theta$, there exists an $\varepsilon>0$ such that
\begin{eqnarray}\label{ratio.lower.bound}
\frac{2m_{\theta,1}(x)}{m_{\theta,2}(x)} &\ge& \frac{1+\varepsilon}{x}
\quad\mbox{for all sufficiently large }x.
\end{eqnarray}
Further, again by Lemma \ref{l:risk} with $\gamma_0=0$, for any fixed $\delta>0$,
\begin{eqnarray}\label{left.tail}
\P\{\xi_\theta(x)\le -\delta x \} &=& O(p(x)/x)\quad\mbox{as }x\to\infty,
\end{eqnarray}
for some decreasing integrable at infinity function $p(x)$,
due to $\E\xi^2<\infty$. 

The bounds \eqref{ratio.lower.bound} and \eqref{left.tail} show that
the conditions (11) and (13) from Theorem 3 in \cite{DKW2013} hold true.
In addition, the Markov chain $\{R_{\theta,n}\}$---the embedded Markov chain for the ruin process with premium rate $v_\theta(z)$---dominates
a similar Markov chain generated by a risk process with constant premium rate $v_c$ which represents a zero-drift random walk.
The latter is null-recurrent and hence satisfying the condition (12) 
from Theorem 3 in \cite{DKW2013}, thus $\limsup_{n\to\infty} R_{\theta,n}=\infty$.
Therefore, Theorem 3 from \cite{DKW2013} applies and
we conclude that the chain $\{R_{\theta,n}\}$ is transient.
Then the original chain $\{R_n\}$ is transient too, 
due to the domination property \eqref{V.compar}.

The convergence to a $\Gamma$-distribution follows from Theorem 4 in \cite{DKW2013}.
\end{proof}

\begin{remark}
It is worth mentioning that the condition \eqref{eq:theta-cond}
is close to be minimal one for $\psi(x)<1$. More precisely, one can show that if 
$$
v(z)\le v_c+\theta/z\quad\text{for all sufficiently large }z
$$
with some 
$$
\theta < \frac{\V\xi+v_c^2\V\tau}{2\E\tau}
$$
then the chain $\{R_n\}$ is recurrent or, equivalently,
$\psi(x)=1$ for all $x>0$.\\
This statement follows by similar arguments applied to a dominating Markov chain
with premium rate $v_\theta(z):=\max(v(z),v_c+\theta/z)$ that satisfies,
for some $\varepsilon>0$,
\begin{eqnarray*}
2zm_{\theta,1}(z) &\le& (1-\varepsilon)m_{\theta,2}(z)\quad\mbox{for all sufficiently large }z,
\end{eqnarray*}
and hence the classical Lamperti criterion (see, e.g. Lamperti \cite{Lamp60}) 
for recurrence of Markov chains applies. 
\end{remark}

\section{Approaching critical premium rate at rate of $\theta/x$}
\label{sec:rate.1x}

\begin{theorem}\label{thm:risk}
Assume \eqref{cond:Rgr0} and the rate of convergence \eqref{risk.2} 
with some $\theta$ satisfying \eqref{eq:theta-cond}, that is,
\begin{eqnarray*}
\theta &>& \frac{\V\xi+v_c^2\V\tau}{2\E\tau}.
\end{eqnarray*}
Set
\begin{eqnarray*}
\rho &:=& \frac{2\theta\E \tau}{\V\xi+v_c^2\V\tau}-1.
\end{eqnarray*}
If
both $\E \tau^2\log(1+\tau)$ and $\E\xi^{\rho+2}$ are finite
then there exist positive constants $c_1$ and $c_2$ such that
\begin{eqnarray*}
\frac{c_1}{(1+x)^\rho} &\le& \psi(x)\ \le\ \frac{c_2}{(1+x)^\rho}
\quad\mbox{for all } x>0.
\end{eqnarray*}
\end{theorem}

These bounds are quite similar to the classical estimates \eqref{eq:Cramer_bound} 
and \eqref{eq:Cramer_asymp}. Indeed, they are universal and only depend  
on a single parameter $\rho$ of the distribution of $(\xi,\tau)$. 
In contrast to the classical Cramer case, the crucial parameter $\rho$ 
is very easy to compute. A further advantage of the bounds in 
Theorem~\ref{thm:risk} is the fact that they are applicable to a wide class 
of claim size distributions: the only restriction is that the moment of order $\rho+2$ 
should be finite; otherwise, the probability of ruin is higher, see Section \ref{sec:heavy-tails}.

By the condition on $\theta$, $\rho>0$. Define 
$$
q(x) := (\rho+1)\min(1,1/x)
$$ 
and
\begin{eqnarray}\label{def.of.R.2}
Q(x) := \int_0^x q(y)dy &\to& \infty\quad\mbox{as }x\to\infty;
\end{eqnarray}
hereinafter we define $Q(x)=0$ for $x<0$.
The increasing function $Q(x)$ is concave on the positive half line
because $q(x)$ is decreasing.
We have, for $c=\rho+1$,
\begin{eqnarray*}
Q(x)\ =\ \int_0^x q(y)dy &=& (\rho+1)\log x+c
\quad\mbox{for all  }x\ge 1,
\end{eqnarray*}
so the function $e^{-Q(x)}$ is integrable at infinity, due to 
$\rho>0$.
It allows us to define the following bounded decreasing function which plays
the most important r\^ole in our analysis of the ruin probabilities:
\begin{eqnarray}\label{def.u.return}
U(x) &:=& \int_x^\infty e^{-Q(y)}dy\quad\mbox{for }x\ge 0;
\end{eqnarray}
and $U(x)=U(0)$ for $x\le 0$. For all $x\ge 1$ we have 
\begin{eqnarray}
\label{eq:QandU}
e^{-Q(x)}= e^{-c}/x^{\rho+1}
\quad\text{and}\quad
U(x)=e^{-c}/\rho x^\rho.
\end{eqnarray}

Let us also define the following auxiliary decreasing functions needed for our analysis.
Without loss of generality we assume that $p_1(x)\le p(x)\le q(x)$ for all $x$, 
where $p_1(x)$ is given by Lemma~\ref{l:risk};
otherwise we can always consider the function $\max(p_1(x),p(x))$ instead of $p(x)$.
Consider the functions $q_+(x):=q(x)+p(x)$ and $q_-(x):=q(x)-p(x)$
and let
\begin{eqnarray}\label{Upm.def}
\nonumber
Q_\pm(x) &:=& \int_0^x q_\pm(y)dy,\\
U_\pm(x) &:=& \int_x^\infty e^{-Q_\pm(y)}dy,\quad x\ge 0,
\end{eqnarray}
and $U_\pm(x)=U_\pm(0)$ for $x\le 0$.
We have $0\le q_-(x)\le q(x)\le q_+(x)$, 
$0\le Q_-(x)\le Q(x)\le Q_+(x)$ and
$U_-(x)\ge U(x)\ge U_+(x)>0$. Since
\begin{eqnarray*}
C_p &:=& \int_0^\infty p(y)dy\quad\mbox{is finite},
\end{eqnarray*}
we have
\begin{eqnarray}\label{equiv.for.R.return}
Q_\pm(x) &=& Q(x)\pm C_p+o(1)\quad\mbox{as }x\to\infty.
\end{eqnarray}
Therefore,
\begin{eqnarray}\label{equiv.for.U.return}
U_\pm(x) \sim e^{\mp C_p}U(x)\sim\frac{e^{\mp C_p}}{\rho}xe^{-Q(x)} \quad\mbox{as }x\to\infty.
\end{eqnarray}

Since $p(x)$ is decreasing and integrable, $p(x)x\to0$ as $x\to\infty$.
We also assume that
\begin{eqnarray}\label{r.p.prime.return}
p'(x)\ =\ O(1/x^2).
\end{eqnarray}
It follows from Lemma \ref{l:g.fin.p.2} that the condition on $p'(x)$
is always satisfied for a properly chosen function $p$.

\begin{lemma}\label{L.Lyapunov.return}
As $x\to\infty$,
\begin{eqnarray}\label{b.for.u.return+}
\E U_+(x+\xi(x))-U_+(x) &=&
p(x)(1+o(1))e^{-Q_+(x)}
\end{eqnarray}
and
\begin{eqnarray}\label{b.for.u.return-}
\E U_-(x+\xi(x))-U_-(x) &=& -p(x)(1+o(1))e^{-Q_-(x)}.
\end{eqnarray}
\end{lemma}

\begin{proof} 
We start with the following decomposition:
\begin{eqnarray}\label{L.harm2.1.return}
\E U_\pm(x+\xi(x))-U_\pm(x)
&=& \E\{U_\pm(x+\xi(x))-U_\pm(x);\ \xi(x)<-x/2\}\nonumber\\
&& +\E\{U_\pm(x+\xi(x))-U_\pm(x);\ |\xi(x)|\le x/2\}\nonumber\\
&&+\E\{U_\pm(x+\xi(x))-U_\pm(x);\ \xi(x)>x/2\}.
\end{eqnarray}
The third term on the right hand side is negative because
$U_\pm$ decreases and it may be bounded below as follows:
\begin{eqnarray}\label{L.harm2.2a.return}
\E\{U_\pm(x+\xi(x))-U_\pm(x);\ \xi(x)>x/2\}
&\ge& -U_\pm(x)\P\{\xi(x)>x/2\}\nonumber\\
&=& o\bigl(p_1(x)e^{-Q_\pm(x)}\bigr),
\end{eqnarray}
due to the upper bound \eqref{risk.8.tail} which implies
$\P\{\xi(x)>x/2\}=o(p_1(x)/x)$,
and due to the relations \eqref{equiv.for.R.return} and \eqref{equiv.for.U.return}.
Further, the first term on the right hand side of \eqref{L.harm2.1.return}
is positive and possesses the following upper bound:
\begin{eqnarray}\label{L.harm2.2b.return}
\E\{U_\pm(x+\xi(x))-U_\pm(x);\ \xi(x)<-x/2\}
&\le& \E\{U_\pm(x+\xi(x));\ \xi(x)<-x/2\}\nonumber\\
&=& o\bigl(p_1(x)e^{-Q_\pm(x)}\bigr),
\end{eqnarray}
due to the upper bound \eqref{cond.xi.le.return} 
and due to the relations \eqref{eq:QandU} and \eqref{equiv.for.U.return}.

To estimate the second term on the right hand side of \eqref{L.harm2.1.return},
we make use of Taylor's expansion:
\begin{eqnarray}\label{L.harm2.2.return}
\lefteqn{\E\{U_\pm(x+\xi(x))-U_\pm(x);\ |\xi(x)|\le x/2\}}\nonumber\\
&&\hspace{15mm} =\ U_\pm'(x)\E\{\xi(x);|\xi(x)|\le x/2\}
+\frac{1}{2} U_\pm''(x)\E\{\xi^2(x);|\xi(x)|\le x/2\}\nonumber\\
&&\hspace{35mm} +\frac{1}{6}\E\bigl\{U_\pm'''(x+\theta\xi(x))\xi^3(x);
|\xi(x)|\le x/2\bigr\}\nonumber\\
&&\hspace{15mm} =\ U_\pm'(x)m_1(x)+\frac{1}{2} U_\pm''(x)m_2(x)\nonumber\\
&&\hspace{20mm}-U_\pm'(x)\E\{\xi(x);|\xi(x)|>x/2\}
-\frac{1}{2} U_\pm''(x)\E\{\xi^2(x);|\xi(x)|>x/2\}\nonumber\\
&&\hspace{35mm} +\frac{1}{6}\E\bigl\{U_\pm'''(x+\theta\xi(x))\xi^3(x);
|\xi(x)|\le x/2\bigr\},
\end{eqnarray}
where $0\le\theta=\theta(x,\xi(x))\le 1$. By the construction of $U_\pm$,
\begin{eqnarray}\label{U.12.prime.return}
U_\pm'(x)=-e^{-Q_\pm(x)},\qquad
U_\pm''(x)=q_\pm(x)e^{-Q_\pm(x)}=(q(x)\pm p(x))e^{-Q_\pm(x)}.
\end{eqnarray}
Then it follows that
\begin{eqnarray}\label{L.harm2.2.1.return}
U_\pm'(x)m_1(x)+\frac{1}{2}U_\pm''(x)m_2(x)
&=& e^{-Q_\pm(x)} \Bigl(-m_1(x)+(q(x)\pm p(x))\frac{m_2(x)}{2}\Bigr)\nonumber\\
&=& \frac{m_2(x)}{2}e^{-Q_\pm(x)}
\biggl(-\frac{2m_1(x)}{m_2(x)}+q(x)\pm p(x)\biggr)\nonumber\\
&=& \pm\frac{m_2(x)}{2}e^{-Q_\pm(x)} p(x) (1+o(1)),
\end{eqnarray}
by Lemma \ref{l:risk} which yields
\begin{eqnarray*}
\frac{2m_1(x)}{m_2(x)}
&=& q(x)+o(p(x)+1/x^2))\quad\mbox{as }x\to\infty.
\end{eqnarray*}
It follows from \eqref{risk.8.tail} and  \eqref{U.12.prime.return} that
\begin{eqnarray}\label{L.harm2.2.5.return}
U_\pm'(x)\E\{\xi(x);|\xi(x)|>x/2\}
+\frac{1}{2} U_\pm''(x)\E\{\xi^2(x);|\xi(x)|>x/2\}
&=& o(p(x)e^{-Q_\pm(x)}).\nonumber\\
\end{eqnarray}
Finally, let us estimate the last term in \eqref{L.harm2.2.return}.
Notice that by the condition \eqref{r.p.prime.return} on the derivative of $p(x)$,
\begin{eqnarray*}
U_\pm'''(x) &=& \bigl(q'(x)\pm p'(x)-(q(x)\pm p(x))^2\bigr)e^{-Q_\pm(x)}\\
&=& O(1/x^2)e^{-Q_\pm(x)},
\end{eqnarray*}
hence, 
\begin{eqnarray*}
U_\pm'''(x+y) &=& O(1/x^2)e^{-Q_\pm(x)}
\end{eqnarray*}
as $x\to\infty$ uniformly for $|y|\le x/2$ which implies
\begin{eqnarray*}
\bigl|\E\bigl\{U_\pm'''(x+\theta\xi(x))\xi^3(x);
|\xi(x)|\le x/2\bigr\}\bigr|
&\le& \frac{c_1}{x^2}\E\bigl\{|\xi^3(x)|;\ |\xi(x)|\le x/2\bigr\}e^{-Q_\pm(x)}.
\end{eqnarray*}
Then, in view of \eqref{cond.3.moment.return},
\begin{eqnarray}\label{full.vs.cond.return}
\bigl|\E\bigl\{U_\pm'''(x+\theta\xi(x))\xi^3(x);\ |\xi(x)|\le x/2\bigr\}\bigr|
&=& o\bigl(p(x)e^{-Q_\pm(x)}\bigr).
\end{eqnarray}
Substituting \eqref{L.harm2.2.1.return}--\eqref{full.vs.cond.return} 
into \eqref{L.harm2.2.return},  we obtain that
\begin{eqnarray}\label{L.harm2.2.2.return}
\E\{U_\pm(x+\xi(x))-U_\pm(x);\ |\xi(x)|\le x/2\}
&=& \pm m_2(x) p(x)(1+o(1))e^{-Q_\pm(x)}.\nonumber\\[-1mm]
\end{eqnarray}
Substituting \eqref{L.harm2.2a.return}---or \eqref{L.harm2.2b.return}---and
\eqref{L.harm2.2.2.return} into \eqref{L.harm2.1.return}
and recalling that $p_1(x)\le p(x)$,
we finally come to the desired conclusions.
\end{proof}

Lemma \ref{L.Lyapunov.return} implies the following result.

\begin{corollary}\label{cor:lyapunov.return}
There exists an $\widehat x$ such that, for all $x>\widehat x$,
\begin{eqnarray*}
\E U_-(x+\xi(x)) &\le& U_-(x),\\
\E U_+(x+\xi(x)) &\ge& U_+(x).
\end{eqnarray*}
\end{corollary}

\begin{proof}[Proof of Theorem \ref{thm:risk}]
The process $U_-(R_n)$ is bounded above by $U_-(0)$.
Let $\widehat x$ be any level guaranteed by the last corollary,
$B=(-\infty,\widehat x]$ and $\tau_B=\min\{n\ge 1:X_n\in B\}$.

By Corollary \ref{cor:lyapunov.return}, $U_-(R_{n\wedge\tau_B})$
is a bounded supermartingale. 
Hence by the optional stopping theorem, for $z>\widehat x$ and $x\in(\widehat x,z)$,
$$
\E_x U_-(R_{\tau_B\wedge\tau_{(z,\infty)}})\ \le\ \E_x U_-(R_0)\ =\ U_-(x).
$$
Letting $z\to\infty$ we conclude that
\begin{eqnarray*}
\E_x\{U_-(R_{\tau_B});\ \tau_B<\infty\} &=& 
\lim_{z\to\infty}\E_x\{U_-(R_{\tau_B});\ \tau_B<\tau_{(z,\infty)}\}\\
&=& \lim_{z\to\infty}\E_x U_-(R_{\tau_B\wedge\tau_{(z,\infty)}})
-\lim_{z\to\infty}\E_x\{U_-(R_{\tau_{(z,\infty)}});\ \tau_B>\tau_{(z,\infty)}\}\\
&\le& U_-(x)-0\ =\ U_-(x).
\end{eqnarray*}
On the other hand, since $U_-$ is decreasing,
\begin{eqnarray*}
\E_x\{U_-(R_{\tau_B});\ \tau_B<\infty\} &\ge& U_-(\widehat x)\P_x\{\tau_B<\infty\}.
\end{eqnarray*}
Therefore,
\begin{eqnarray}\label{tau.B.infty.U-}
\P_x\{\tau_B<\infty\} &\le& \frac{U_-(x)}{U_-(\widehat x)},
\end{eqnarray}
which implies, by \eqref{equiv.for.U.return}, that, 
for some constant $c_2<\infty$,
$$
\P_x\{R_n\le \widehat x\mbox{ for some }n\}\ \le\ c_2U(x)
\quad\mbox{for all }x>\widehat x.
$$
Thus,
\begin{eqnarray*}
\P_x\{R_n\le 0\mbox{ for some }n\}
&\le& 
\P_x\{R_n\le \widehat x\mbox{ for some }n\}\\ 
&\le& c_2U(x)
\quad\mbox{for all }x>\widehat x.
\end{eqnarray*}
This gives the desired upper bound.

On the other hand, the process $\{U_+(R_{n\wedge\tau_B})\}$ 
is a bounded submartingale
due to the lower bound provided by Corollary \ref{cor:lyapunov.return}.
Hence again by the optional stopping theorem, for $x>x_0$,
$$
\E_x\{U_+(R_{\tau_B});\ \tau_B<\infty\}\ \ge\ \E_x U_+(R_0)\ =\ U_+(x).
$$
On the other hand, since $U_+$ is bounded by $U_+(0)$,
$$
\E_x\{U_+(R_{\tau_B});\ \tau_B<\infty\}\
\le\ U_+(0)\P_x\{\tau_B<\infty\}.
$$
This allows us to deduce a lower bound
\begin{eqnarray*}
\P_x\{\tau_B<\infty\} &\ge& \frac{U_+(x)}{U_+(0)},
\end{eqnarray*}
which completes the proof of the lower bound, 
for some constant $c_1>0$,
$$
\P_x\{R_n\le \widehat x\mbox{ for some }n\}\ \ge\ c_1U(x)
\quad\mbox{for all }x>\widehat x,
$$
due to \eqref{equiv.for.U.return}.
To complete the proof of the lower bound it remains to refer to the
arguments in \eqref{lower.hat.0}.
\end{proof}

\section{Approaching critical premium rate at rate of $\theta/x^\alpha$}
\label{sec:rate.hx}

In this section we consider the case \eqref{risk.2.hx} with some $\alpha\in(0,1)$.
In order to understand the asymptotic behaviour of the ruin probability
under this rate of approaching the critical value $v_c$, we first derive asymptotic estimates
for the moments of $V_x(\tau)-x$. Define
\begin{eqnarray*}
\gamma:=\min\{k\ge1:\alpha k>1\}.
\end{eqnarray*}

\begin{lemma}\label{lem:risk.2.V}
Let $\E\tau^\gamma<\infty$ and there exists an $x_0\ge0$ such that
\begin{eqnarray}\label{risk.2.v12}
v_-(x) &\le& v(x)\ \le\ v_+(x)\quad\mbox{for all }x\ge x_0,
\end{eqnarray}
where both $v_-(x)$ and $v_+(x)$ are decreasing functions on
$[x_0,\infty)$. Then, for all $k\le\gamma$,
\begin{eqnarray}\label{risk.2.bounds}
\E\tau^k \left(v_-(x+\tau v_+(x))\right)^k\ \le\ \E(V_x(\tau)-x)^k &\le& v_+^k(x)\E\tau^k,\quad x\ge x_0.
\end{eqnarray}
If, in addition, $\E\tau^{\gamma+1-\alpha}<\infty$ and 
\eqref{risk.2.hx} holds true, then there exists an integrable
decreasing function $p_1(x)$ such that, for all $k\le\gamma$,
\begin{eqnarray}\label{risk.2.1}
\E(V_x(\tau)-x)^k &=&
(v_c+\theta/x^\alpha)^k\E\tau^k+O(p_1(x))
\quad\mbox{as }x\to\infty.
\end{eqnarray}
\end{lemma}

\begin{proof}
Fix some $x\ge x_0$. Due to \eqref{risk.2.v12}, $v(z)\le v_+(x)$ for all $z\ge x$. Hence,
\begin{eqnarray}\label{risk.2.v2}
\nonumber
V_x(t) &=& x+\int_0^t v(V_x(s))ds\\
&\le& x+\int_0^t v_+(x)ds
\ =\ x+tv_+(x),
\end{eqnarray}
and the inequality on the right hand side of \eqref{risk.2.bounds} follows.
It follows from the left hand side inequality in \eqref{risk.2.v12}
and from the last upper bound for $V_x(t)$ that
\begin{eqnarray}\label{risk.2.v1}
V_x(t)-x &\ge& \int_0^t v_-(V_x(t))ds
\ \ge\ tv_-(x+tv_+(x)),
\end{eqnarray}
and the left hand side bound in \eqref{risk.2.bounds} is proven.

Owing to \eqref{risk.2.hx}, $v(z)$ is sandwiched between the two
eventually decreasing functions $v_{\pm}(z):=v_c+\theta/z^\alpha\pm p(z)$.
Therefore, applying the right hand side bound in \eqref{risk.2.bounds} we get
\begin{eqnarray}\label{risk.2.above}
\E(V_x(\tau)-x)^k &\le& (v_c+\theta/x^\alpha+p(x))^k\E\tau^k\nonumber\\
&=& (v_c+\theta/x^\alpha)^k\E\tau^k+O(p(x))\quad\mbox{as }x\to\infty.
\end{eqnarray}
From the lower bound in \eqref{risk.2.bounds} we deduce that, for all $k\le\gamma$,
\begin{eqnarray*}
\E(V_x(\tau)-x)^k &\ge&
\E\tau^k\Bigl(v_c+\frac{\theta}{(x+\tau v_+(x))^\alpha}-p(x)\Bigr)^k\\
&\ge& \E\tau^k\Bigl(v_c+\frac{\theta}{(x+\overline v\tau)^\alpha}\Bigr)^k+O(p(x)),
\quad \overline v=\sup_z v(z).
\end{eqnarray*}
By the inequality $1/(1+y)^\alpha\ge 1-\alpha y\wedge 1$,
we infer that, for $c_2=\alpha\overline v$,
\begin{eqnarray*}
\frac{1}{(x+\overline vt)^\alpha}
&\ge& \frac{1}{x^\alpha}\Bigl(1-\frac{c_2t}{x}\wedge 1\Bigr).
\end{eqnarray*}
Therefore, for all $k\le\gamma$,
\begin{eqnarray}\label{risk.2.3}
\E(V_x(\tau)-x)^k &\ge&
\E\tau^k\Bigl(v_c+\frac{\theta}{x^\alpha}
-\frac{c_2\theta\tau}{x^{\alpha+1}}\I\{\tau\le x/c_2\}
-\frac{\theta}{x^\alpha}\I\{\tau>x/c_2\}\Bigr)^k
+O(p(x))\nonumber\\
&\ge& \Bigl(v_c+\frac{\theta}{x^\alpha}\Bigr)^k\E\tau^k
-\frac{c_3}{x^\alpha}\E\{\tau^k;\ \tau>x/c_2\}\nonumber\\
&&\hspace{10mm} -c_3\sum_{j=1}^k\frac{1}{x^{j(\alpha+1)}}\E\{\tau^{k+j};\ \tau\le x/c_2\}
-c_3p(x),
\hspace{5mm}
\end{eqnarray}
for some $c_3<\infty$. Then, due to the integrability of $p(x)$,
in order to prove that
\begin{eqnarray}\label{risk.2.below}
\E(V_x(\tau)-x)^k &\ge& (v_c+\theta/x^\alpha)^k\E\tau^k-p_1(x)
\end{eqnarray}
for some decreasing integrable function $p_1(x)$, it suffices to show that
$$
x^{-\alpha}\E\{\tau^\gamma;\ \tau>x\}
$$
and
$$
x^{-j(\alpha+1)}\E\{\tau^{\gamma+j};\ \tau\le x\}
$$
are bounded by decreasing integrable at infinity functions.
Indeed, the integral of the first function---which decreases itself---is finite 
due to the finiteness of the $(\gamma+1-\alpha)$ moment of $\tau$. 
Concerning the second function, first notice that
\begin{eqnarray*}
x^{-j(\alpha+1)}\E\{\tau^{\gamma+j};\ \tau\le x\} &\le&
\frac{\E\{\tau^{\gamma+1};\ \tau\le x\}}{x^{1+\alpha}},\quad j\ge 1.
\end{eqnarray*}
The right hand side is bounded by a decreasing integrable at infinity function
due to the moment condition on $\tau$ and  Lemma \ref{l:p.V.maj.o}.
So, \eqref{risk.2.below} is proven
which together with \eqref{risk.2.above} completes the proof.
\end{proof}

\begin{proposition}\label{prop:risk.2}
Assume the rate of convergence \eqref{risk.2.hx}.
If both $\E\tau^{1+\gamma}$ and $\E\xi^{1+\gamma}$ are finite,
then, for all $k\le\gamma$,
\begin{eqnarray*}
m_k(x) &=&
\sum_{j=0}^k\frac{a_{k,j}}{x^{\alpha j}} +O(x^{\alpha(k-1)}p_2(x))
\quad\mbox{as }x\to\infty,
\end{eqnarray*}
where $p_2(x)$ is a decreasing integrable at infinity function and
\begin{eqnarray*}
a_{k,j} &:=& {k\choose j}\theta^j \E\tau^j(v_c\tau-\xi)^{k-j},
\quad j\le k\le\gamma.
\end{eqnarray*}
In addition,
\begin{eqnarray}\label{eq:x^alpha-tail}
\E\{|\xi^k(x)|;\ |\xi(x)|>x^\alpha\} &=& o(x^{\alpha(k-1)}p_2(x))
\quad\mbox{as }x\to\infty.
\end{eqnarray}
\end{proposition}

\begin{proof}
It follows from the definition of $\xi(x)$ that
\begin{eqnarray*}
\E\xi^k(x) &=&
\E(V_x(\tau)-x-\xi)^k
\ =\ \sum_{i=0}^k{k\choose i}\E(V_x(\tau)-x)^i\E(-\xi)^{k-i}.
\end{eqnarray*}
Applying Lemma \ref{lem:risk.2.V}, we then obtain
\begin{eqnarray*}
m_k(x)\ :=\ \E\xi^k(x) &=& 
\sum_{i=0}^k{k\choose i}\Bigl(v_c+\frac{\theta}{x^\alpha}\Bigr)^i
\E\tau^i\E(-\xi)^{k-i}+O(p_1(x))\\
&=& \sum_{i=0}^k{k\choose i}\E\tau^i\E(-\xi)^{k-i}
\sum_{j=0}^i{i\choose j}v_c^{i-j}\Bigl(\frac{\theta}{x^\alpha}\Bigr)^j+O(p_1(x))\\
&=:& \sum_{j=0}^k\frac{a_{k,j}}{x^{\alpha j}}+O(p_1(x))\quad\mbox{as }x\to\infty,
\end{eqnarray*}
where
\begin{eqnarray*}
a_{k,j} &:=& {k\choose j}\theta^j 
\sum_{i=j}^k{k-j\choose i-j}\E\tau^i\E(-\xi)^{k-i}v_c^{i-j}\\
&=& {k\choose j}\theta^j \E\sum_{i=0}^{k-j}{k-j\choose i}\tau^{i+j}(-\xi)^{k-j-i}v_c^i\\
&=& {k\choose j}\theta^j \E\tau^j(v_c\tau-\xi)^{k-j}.
\end{eqnarray*}

It is immediate from \eqref{risk.2.v2} that $V_x(\tau)-x\le \overline v\tau$
where $\overline v=\sup_z v(z)$. Then
\begin{eqnarray*}
\E\{|\xi^k(x)|;\ |\xi(x)|>x^\alpha\}
&\le& \E\{(V_x(\tau)-x)^k;\ V_x(\tau)-x>x^\alpha\}+\E\{\xi^k;\xi>x^\alpha\}\\
&\le& \overline v^k\E\{\tau^k;\ \tau>x^\alpha/\overline v\}
+\E\{\xi^k;\ \xi>x^\alpha\}.
\end{eqnarray*}
Since $\E\tau^{\gamma+1}<\infty$, for all $k\le\gamma$,
\begin{eqnarray*}
x^{-\alpha(k-1)}\E\{\tau^k;\ \tau>x^\alpha/\overline v\} 
&=& o\left(\frac{1}{x^{\alpha(k-1)} x^{\alpha(\gamma+1-k)}}\right)\\
&=& o\left(\frac{1}{x^{\alpha\gamma}}\right)\quad\mbox{as }x\to\infty.
\end{eqnarray*}
By the definition of the $\gamma$, 
$\alpha\gamma>1$. The function $1/x^{\alpha\gamma}$ is integrable at infinity.
The same arguments work for $\xi$, so the value of
$x^{-\alpha(k-1)}\E\{|\xi^k(x)|;\ |\xi(x)|>x^\alpha\}$ 
is bounded by a decreasing integrable at infinity function,
 and the proof is complete.
\end{proof}

Now we state the main result in this section.
\index{Cram\'er--Lundberg!classical model!ruin probability bounds}

\begin{theorem}\label{thm:risk.2}
Assume  \eqref{cond:Rgr0} and the rate of convergence \eqref{risk.2.hx}.
Let $\E\tau^{\gamma+1}<\infty$ and $\E e^{r\xi^{1-\alpha}}<\infty$ for some
$$
r\ >\frac{r_1}{1-\alpha},
$$
where
\begin{eqnarray}\label{eq:r1}
r_1 &:=& \frac{2\theta\E\tau}{\V\xi+v_c^2\V\tau}.
\end{eqnarray}
Then there exist constants $r_2$, $r_3$, \ldots, $r_{\gamma-1}\in\R$,
and $0<C_1<C_2<\infty$ such that
\begin{itemize}
\item[(i)] if $\alpha=1/(\gamma-1)$ for an integer $\gamma\ge 2$,
then, for $x>1$,
\begin{eqnarray}\label{risk.2.6}
\frac{C_1x^\alpha}{x^{r_{\gamma-1}}}
\exp\Biggl\{-\sum_{j=1}^{\gamma-2}\frac{r_j}{1-\alpha j} x^{1-\alpha j}\Biggr\}
\ \le\ \psi(x)\ \le\
\frac{C_2x^\alpha}{x^{r_{\gamma-1}}}
\exp\Biggl\{-\sum_{j=1}^{\gamma-2}\frac{r_j}{1-\alpha j} x^{1-\alpha j}\Biggr\},
\nonumber\\
\end{eqnarray}
\item[(ii)] if $\alpha<1/(\gamma-1)$ then
\begin{eqnarray}\label{risk.2.6*}
C_1x^\alpha\exp\Biggl\{-\sum_{j=1}^{\gamma-1}\frac{r_j}{1-\alpha j} x^{1-\alpha j}\Biggr\}
\ \le\ \psi(x)\ \le\
C_2x^\alpha\exp\Biggl\{-\sum_{j=1}^{\gamma-1}\frac{r_j}{1-\alpha j} x^{1-\alpha j}\Biggr\}.
\nonumber\\
\end{eqnarray}
\end{itemize}
\end{theorem}

\begin{proof}
We first show that there exist constants
$r_1,r_2,\ldots,r_{\gamma-1}$ such that
$$
q(x):=\sum_{j=1}^{\gamma-1}\frac{r_j}{(b+x)^{\alpha j}}
$$
satisfies 
\begin{equation}
\label{eq:q-moments}
-m_1(x)+\sum_{j=2}^\gamma(-1)^j\frac{m_j(x)}{j!}q^{j-1}(x)
=o(p_3(x)),
\end{equation}
where $p_3$ is a decreasing integrable function and $b$ is a positive number.

We can determine all these numbers recursively.
Indeed, as proven in Proposition~\ref{prop:risk.2},
\begin{eqnarray*}
m_1(x) &=& \frac{\theta\E\tau}{x^\alpha}+o(p_2(x))\quad\mbox{as }x\to\infty
\end{eqnarray*}
and
\begin{eqnarray*}
m_2(x) &=& \V\xi+v_c^2\V\tau+O(x^{-\alpha}+x^\alpha p_2(x))
\quad\mbox{as }x\to\infty.
\end{eqnarray*}
For $r_1$ defined defined in \eqref{eq:r1},
\begin{eqnarray*}
-m_1(x)+\sum_{j=2}^{\gamma} (-1)^j\frac{m_j(x)}{j!}q^{j-1}(x)
&=& O(x^{-2\alpha}+p_2(x))\quad\mbox{as }x\to\infty,
\end{eqnarray*}
for any choice of $r_2$, $r_3$, \ldots, $r_{\gamma-1}$.
Then we can choose $r_2$ such that
the coefficient of $x^{-2\alpha}$ is also zero, and so on.
It is clear that the numbers $r_1$, $r_2$, \ldots, $r_{\gamma-1}$ 
do not depend on the parameter $b$. 
Therefore, we can take $b$ so large that the function $q(x)$ is decreasing on $[0,\infty)$.

As in the previous section, we define 
$$
Q(x)=\int_0^x q(y)dy
\quad\text{and}\quad 
U(x)=\int_x^\infty e^{-Q(y)}dy,\ x\ge0.
$$
For $x<0$ we set $U(x)=U(0)$. It is immediate from the definition of 
$q(x)$ that 
\begin{eqnarray*}
Q(x) &=&\int_0^x \sum_{j=1}^{\gamma-1}\frac{r_j}{(b+z)^{\alpha j}} dz,\quad x\ge0.
\end{eqnarray*}
and
\begin{eqnarray*}
U(x) &=& \int_x^\infty \exp\Bigl\{-
\int_0^y \sum_{j=1}^{\gamma-1}\frac{r_j}{(b+z)^{\alpha j}} dz \Bigr\}dy,\quad x\ge0.
\end{eqnarray*}
We define also
$$
q_\pm(x)=q(x)\pm p(x),\quad 
Q_\pm(x)=\int_0^x q_\pm(y)dy
\quad\text{and}\quad 
U_\pm(x)=\int_x^\infty e^{-Q(y)}dy.
$$
We further assume that, for all $1\le k\le\gamma-1$,
\begin{eqnarray}\label{4.r.prime.gen.hy}
q^{(k)}(x) = o(q^\gamma(x)), &&
p^{(k)}(x) = o(q^\gamma(x))\quad\mbox{as }x\to\infty
\end{eqnarray}
and 
\begin{eqnarray}\label{4.r.gamma.hy}
q^\gamma(x) &=& o(p(x))\quad\mbox{as }x\to\infty.
\end{eqnarray}
If $q(x)\sim c/x^\alpha$ where $\gamma\alpha<2$,
then it follows from Lemma \ref{l:g.fin.p.2} that the condition on
the derivatives of $p(x)$
is always satisfied for a properly chosen function $p$,
so the condition \eqref{4.r.prime.gen.hy} on the derivatives
of $p$ does not restrict generality under this specific choice of $r(x)$.

It is clear that 
$$
U_\pm(x)\sim e^{\mp C_p}U(x)\quad\text{as }x\to\infty.
$$ 
Noting that
$$
\frac{U'(x)}{\left(\frac{1}{q(x)}e^{-Q(x)}\right)'}
=\frac{-e^{-Q(x)}}{(-q'(x)/q^2(x)-1)e^{-Q(x)}} \to 1
\quad\text{as }x\to\infty
$$
and applying the L'H\^{o}pital rule, we conclude that,
as $x\to\infty$,
\begin{equation}\label{eq:U-asymp}
U(x)\sim \frac{e^{-Q(x)}}{q(x)}
\quad\text{and}\quad
U_\pm(x)\sim e^{\mp C_p}\frac{e^{-Q(x)}}{q(x)}
\sim \frac{e^{-Q_\pm(x)}}{q(x)}.
\end{equation}

\begin{lemma}\label{L.Lyapunov.return.hy}
As $x\to\infty$, we have the following estimates:
\begin{eqnarray}\label{b.for.u.return.1.hy}
\E U_+(x+\xi(x))-U_+(x) &=&
\frac{\V\xi+v_c^2\V\tau+o(1)}{2}p(x)e^{-R_+(x)},\\
\label{b.for.u.return.2.hy}
\E U_-(x+\xi(x))-U_-(x) &=& -\frac{\V\xi+v_c^2\V\tau+o(1)}{2}p(x)e^{-R_-(x)}.
\end{eqnarray}
\end{lemma}

\begin{proof}
We start with the following decomposition:
\begin{eqnarray}\label{L.harm2.1.return1}
\E U_\pm(x+\xi(x))-U_\pm(x)
&=& \E\{U_\pm(x+\xi(x))-U_\pm(x);\ \xi(x)<-x^\alpha\}\nonumber\\
&&+\E\{U_\pm(x+\xi(x))-U_\pm(x);\ |\xi(x)|\le x^\alpha\}\nonumber\\
&& +\E\{U_\pm(x+\xi(x))-U_\pm(x);\ \xi(x)>x^\alpha\}.
\end{eqnarray}
The third term on the right hand side is negative because
$U_\pm$ decreases and it may be bounded below as follows:
\begin{eqnarray}\label{L.harm2.2a.return1}
\E\{U_\pm(x+\xi(x))-U_\pm(x);\ \xi(x)>x^\alpha\}
&\ge& -U_\pm(x)\P\{\xi(x)>x^\alpha\}\nonumber\\
&=& o\bigl(p_2(x)e^{-Q_\pm(x)}\bigr),
\end{eqnarray}
due to the upper bound \eqref{eq:x^alpha-tail} which implies
$\P\{\xi(x)>x^\alpha\}=o(p_1(x)/x^\alpha)$,
and due to the relation \eqref{eq:U-asymp}.

Further, the first term on the right hand side of \eqref{L.harm2.1.return1} is positive. 
To obtain an upper bound for that expectation we first notice that, 
due to the fact that $Q(z)$ is monotone increasing,
$$
U_\pm(x-y)-U_\pm(x)=\int_{x-y}^xe^{-Q_\pm(z)}dz
\le e^{2C_p}\int_{x-y}^xe^{-Q(z)}dz
\le e^{2C_p}ye^{-Q(x-y)}.
$$
Since $q(x)$ is chosen to be decreasing, $Q(z)$ is concave and, consequently, 
$$
Q(x-y)\ge Q(x)-Q(y).
$$
Using this inequality we obtain 
$$
U_\pm(x-y)-U_\pm(x)\le e^{2C_p}ye^{-Q(x)}e^{Q(y)}
$$
and 
\begin{eqnarray*}
\lefteqn{\E\{U_\pm(x+\xi(x))-U_\pm(x);\ \xi(x)<-x^\alpha\}}\\
&&\hspace{10mm}\le\ e^{2C_p}e^{-Q(x)}\E\{-\xi(x)e^{Q(\xi(x))};\xi(x)<-x^\alpha\}\\
&&\hspace{20mm}\le\ e^{2C_p}e^{-Q(x)}\E\{\xi e^{Q(\xi)};\xi>x^\alpha\}.
\end{eqnarray*}
The moment assumption on $\xi$ implies that the decreasing function 
$\E\{\xi e^{Q(\xi)};\xi>x^\alpha\}$ is integrable at infinity. As a result we have 
\begin{eqnarray}\label{L.harm2.2b.return1}
\E\{U_\pm(x+\xi(x))-U_\pm(x);\ \xi(x)<-x^\alpha\}
&=& o\bigl(p_1(x)e^{-Q_\pm(x)}\bigr).
\end{eqnarray}

To estimate the second term on the right hand side of \eqref{L.harm2.1.return1},
we make use of Taylor's expansion with $\gamma+1$ terms:
\begin{eqnarray}\label{4.L.harm2.2.hy}
\lefteqn{\E\{U_\pm(x+\xi(x))-U_\pm(x);\ |\xi(x)|\le x^\alpha\}}\nonumber\\
&=& \sum_{k=1}^\gamma\frac{U_\pm^{(k)}(x)}{k!} 
\E\{\xi^k(x);|\xi(x)|\le x^\alpha\}
\nonumber\\
& &\hspace{1cm}+\E\Bigl\{\frac{U_\pm^{(\gamma+1)}(x+\theta\xi(x))}{(\gamma+1)!}\xi^{\gamma+1}(x);\
|\xi(x)|\le x^\alpha\Bigr\},\nonumber\\
&=& \sum_{k=1}^\gamma\frac{U_\pm^{(k)}(x)}{k!}m_k(x)
-\sum_{k=1}^\gamma\frac{U_\pm^{(k)}(x)}{k!} 
\E\{\xi^k(x);|\xi(x)|> x^\alpha\}
\nonumber\\
& &\hspace{1cm}+\E\Bigl\{\frac{U_\pm^{(\gamma+1)}(x+\theta\xi(x))}{(\gamma+1)!}\xi^{\gamma+1}(x);\
|\xi(x)|\le x^\alpha\Bigr\},
\end{eqnarray}
where $0\le\theta=\theta(x,\xi(x))\le 1$. By the construction of $U_\pm$,
\begin{eqnarray}\label{4.U.12.prime.hy}
U_\pm'(x)=-e^{-Q_\pm(x)},\qquad
U_\pm''(x)=q_\pm(x)e^{-Q_\pm(x)}=(q(x)\pm p(x))e^{-Q_\pm(x)},
\end{eqnarray}
and, for $k=3$, \ldots, $\gamma+1$,
\begin{eqnarray*}
U_\pm^{(k)}(x) = -(e^{-Q_\pm(x)})^{(k-1)}
&=& (-1)^k\bigl(q_\pm^{k-1}(x)+o(p(x))\bigr)e^{-Q_\pm(x)}
\quad\mbox{as }x\to\infty,
\end{eqnarray*}
where the remainder terms in the parentheses on the right
are of order $o(p(x))$ by the conditions \eqref{4.r.prime.gen.hy}
and \eqref{4.r.gamma.hy}.
By the definition of $q_\pm(x)$,
\begin{eqnarray*}
q_\pm^{k-1}(x) &=& (q(x)\pm p(x))^{k-1}= q^{k-1}(x)+o(p(x))
\quad\mbox{for all }k\ge 3,
\end{eqnarray*}
which implies the relation
\begin{eqnarray}\label{4.U.k.prime.hy}
U_\pm^{(k)}(x) &=& (-1)^k\bigl(q^{k-1}(x)+o(p(x))\bigr)e^{-Q_\pm(x)}
\quad\mbox{as }x\to\infty.
\end{eqnarray}
From these equalities we get 
$|U_\pm^{(k)}(x)|\le Cx^{-\alpha(k-1)}e^{-Q_\pm(x)}$,
Combining this with \eqref{eq:x^alpha-tail}, we obtain
\begin{equation}
 \label{eq:sum_of_tais}
 \sum_{k=1}^\gamma\frac{U_\pm^{(k)}(x)}{k!} 
\E\{\xi^k(x);|\xi(x)|> x^\alpha\}=o\left(p_2(x)e^{-Q_\pm(x)}\right)
\quad\mbox{as }x\to\infty.
\end{equation}

It follows from the equalities \eqref{4.U.12.prime.hy}
and \eqref{4.U.k.prime.hy} that
\begin{eqnarray}\label{4.L.harm2.2.1.hy}
\lefteqn{\sum_{k=1}^\gamma \frac{U_\pm^{(k)}(x)}{k!}m_k(x)}\nonumber\\
&=& e^{-Q_\pm(x)} \biggl(\sum_{k=1}^\gamma
(-1)^k\frac{r^{k-1}(x)}{k!}m_k(x)
+o(p(x))\pm p(x)\frac{m_2(x)}{2}\biggr)\nonumber\\
&=& e^{-Q_\pm(x)} \biggl(o(p(x))\pm p(x)\frac{m_2(x)}{2}\biggr)
\quad\mbox{as }x\to\infty,
\end{eqnarray}
by the equality \eqref{eq:q-moments}.
Owing to the condition \eqref{4.r.prime.gen.hy}
on the derivatives of $r(x)$ and \eqref{4.r.gamma.hy},
\begin{eqnarray*}
U_\pm^{(\gamma+1)}(x) &=&
(-1)^{\gamma+1}(q^\gamma(x)+o(q^\gamma(x))) e^{-Q_\pm(x)}
\quad\mbox{as }x\to\infty.
\end{eqnarray*}
Then, the last term in
\eqref{4.L.harm2.2.hy} possesses the following bound:
\begin{eqnarray*}
\lefteqn{\Bigl|\E\Bigl\{\frac{U_\pm^{(\gamma+1)}(x+\theta\xi(x))}{(\gamma+1)!}\xi^{\gamma+1}(x);\
|\xi(x)|\le x^\alpha\Bigr\}\Bigr|}\\
&&\hspace{40mm} = O\bigl(q^\gamma(x)e^{-Q_\pm(x)}\bigr)
\E\bigl\{|\xi(x)|^{\gamma+1};\ |\xi(x)|\le x^\alpha\bigr\}\\
&&\hspace{40mm} = o\bigl(p(x)e^{-Q_\pm(x)}\bigr)
\quad\mbox{as }x\to\infty,
\end{eqnarray*}
by the condition \eqref{4.r.gamma.hy}. Therefore, it follows from
\eqref{4.L.harm2.2.hy}, \eqref{eq:sum_of_tais} and \eqref{4.L.harm2.2.1.hy} that
\begin{eqnarray*}
\lefteqn{\E\{U_\pm(x+\xi(x))-U_\pm(x);\ |\xi(x)|\le x^\alpha\}}\nonumber\\
&&\hspace{20mm} =\pm p(x)\frac{m^{[s(x)]}_2(x)}{2} e^{-Q_\pm(x)}
+o\bigl(p(x)e^{-Q_\pm(x)}\bigr)\quad\mbox{as }x\to\infty.
\end{eqnarray*}
Together with \eqref{L.harm2.2a.return1}, \eqref{L.harm2.2b.return1},
and \eqref{L.harm2.1.return1} this completes the proof.
\end{proof}
The remaining part of the proof repeats literally the final
part of the proof of Theorem~\ref{thm:risk} and we omit it.
\end{proof}

\section{Heavy-tailed claim sizes }
\label{sec:heavy-tails}

In this section we study the case where the distribution of the claim size is so heavy 
that the moment conditions in the theorems proved above are not met.  

We assume that $v(x)$ converges towards $v_c$ at rate 
$\theta/x$ and that the distribution of $\xi$ is regularly varying
at infinity with index $-(\beta+2)$ for some $\beta\in(0,\rho)$.  
Then $\E\xi^{\rho+2}$ is infinite and, consequently, Theorem~\ref{thm:risk} does not apply.

\begin{theorem}\label{thm:heavy-tails}
Assume the rate of convergence \eqref{risk.2} with some $\theta$ satisfying 
\eqref{eq:theta-cond}.
Assume also that $\E\tau^2\log(1+\tau)<\infty$ and that 
\begin{equation}\label{eq:xi-heavy}
\P\{\xi>x\}=x^{-2-\beta}L(x)
\end{equation}
for some slowly varying at infinity function $L(x)$ and $\beta\in(0,\rho)$. 
Then there exist constants $C_1$ and $C_2$ such that 
$$
C_1 \int_x^\infty y\P\{\xi>y\}dy
\le \psi(x)\le 
C_2 \int_x^\infty y\P\{\xi>y\}dy
\quad\mbox{for all }x>0.
$$
\end{theorem}

Under the condition \eqref{eq:xi-heavy}, by Karamata's theorem,
$$
\int_x^\infty y\P\{\xi>y\}dy
\sim\frac{1}{\beta}x^2\P\{\xi>x\}\quad\mbox{as }x\to\infty.
$$
Therefore, the claim of Theorem~\ref{thm:heavy-tails}
can be reformulated in the following way:
$$
\widehat{C}_1 x^2 \P\{\xi>x\}
\le \psi(x)\le 
\widehat C_2 x^2 \P\{\xi>x\}.
$$
Notice that, for the classical ruin process with constant premium rate 
and with claim size of subexponential type,
$\psi(x)$ is asymptotically equivalent to the integral $\int_x^\infty \P\{\xi>y\}dy$
(see e.g. \cite[Section 5.11]{FKZ}). 
So, the main difference between our case and the classical one is that the 
probability of ruin is higher in our case owing to 
the additional weight $y$ in the integral, which is not surprising 
and reflects the fact that our system is close to a critical one, 
$v(y)\to v_c$ as $y\to\infty$.

Notice that the condition \eqref{cond:Rgr0} follows by \eqref{eq:xi-heavy}.

The proof of Theorem \ref{thm:heavy-tails} is split into two parts,
where we derive the upper and lower bounds.
For both, we need the following result on the left tail distribution
of the jumps of the chain $\{R_n\}$.

\begin{lemma}
\label{lem:g-tail}
If the distribution of $\xi$ is long-tailed, that is, if
$$
\lim_{x\to\infty}\frac{\P\{\xi>x+u\}}{\P\{\xi>x\}}=1
\quad\text{for any fixed }u,
$$
then, uniformly for all $x\ge 0$,
\begin{eqnarray*}
\P\{\xi(x)< -y\} &\sim& \P\{\xi>y\}\quad\text{as }y\to\infty.
\end{eqnarray*}
\end{lemma}

\begin{proof}
Using the equality $\xi(x)=V_x(\tau)-x-\xi$, 
we get the following upper bound
\begin{eqnarray*}
\P\{\xi(x)< -y\} &=& \P\{\xi-(V_x(\tau)-x)>y\}\\
&\le& \P\{\xi>y\}.
\end{eqnarray*}
For a lower bound, let us notice that, for any fixed $u$,
\begin{eqnarray*}
\P\{\xi(x)\le -y\} &\ge& \P\{\xi>y+u\} \P\{V_x(\tau)-x<u\}\\
&\sim& \P\{\xi>y\} \P\{V_x(\tau)-x<u\}\quad\mbox{as }y\to\infty,
\end{eqnarray*}
due to the long-tailedness of the distribution of $\xi$. 
Also, by the stochastic boundedness of the family
of random variables $\{V_x(\tau)-x,\ x\ge 0\}$,
\begin{eqnarray*}
\inf_x\P\{V_x(\tau)-x<u\} &\to& 1\quad\mbox{as }u\to\infty,
\end{eqnarray*}
which implies the following lower bound, uniformly for all $x\ge 0$,
\begin{eqnarray*}
\P\{\xi(x)\le -y\} &\ge& \P\{\xi>y\} (1+o(1))\quad\mbox{as }y\to\infty,
\end{eqnarray*}
hence the desired result.
\end{proof}

\subsection{Proof of the upper bound}
As in the previous sections, we analyse the behaviour of the chain $R_n=R(T_n)$, $n\ge0$. 
In order to understand the impact of large claim sizes on ruin probabilities
from the point of view of an upper bound,
we introduce an auxiliary chain with jumps truncated below. 
For every $x\ge0$ we define jump $\widetilde{\xi}(x)$ as follows:
$$
\P\{\widetilde{\xi}(x)\in B\}:=\P\{\xi(x)\in B\mid \xi(x)\ge -x/2\},
\quad B\in\mathcal{B}(\R).
$$
Let $\{\widetilde{R}_n\}$ be a Markov chain with jumps $\widetilde{\xi}(x)$. 
The connection between $\{\widetilde{R}_n\}$ and $\{R_n\}$ is described in the next lemma.

\begin{lemma}\label{lemma:mengenAn}
Set $A_n := \{ \xi(R_k) \ge  -R_k/2 \text{ for all } k < n \}$. 
Then, for all Borel sets $B_1, B_2,\ldots, B_n$ we have
\begin{eqnarray*}
\P_x \{R_1 \in B_1,\ldots, R_n \in B_n ; A_n\} &=& 
\E_x \biggl\{ \prod_{k=0}^{n-1} g(\widetilde R_k);\ \TR_1 \in B_1,\ldots, \TR_n \in B_n \biggr\},
\end{eqnarray*}
where 
\begin{eqnarray*}
    g(x) := \mathbb{P} \{\xi(x)\ge -x/2\} \in (0,1).
\end{eqnarray*}
\end{lemma}
\begin{proof}
We use the induction in $n$. If $n=1$, then
\begin{eqnarray*}
\P_x \{R_1 \in B_1 ; A_1\} 
&=& \P \{x + \xi(x) \in B_1 , \xi(x) > -x/2\} \\
&=& g(x) \P \{x+\widetilde{\xi}(x) \in B_1\}\\
&=& \E_x \{g(\TR_0);\TR_1 \in B_1\}.
\end{eqnarray*}
For the induction step $n-1\to n$ it suffices to apply the Markov property:
\begin{eqnarray*}
\lefteqn{\P_x\{R_1 \in B_1, ..., R_n \in B_n ; A_n\}}\\
&=& \int_{B_{n-1}}\P_x\{R_1 \in B_1,\ldots, R_{n-1}\in dy; A_{n-1}\}
\P\{y + \xi(y) \in B_n, \xi(y)\ge  -y/2\} \\
&=& \int_{B_{n-1}} \E_x 
\biggl[\prod_{k=0}^{n-2} g(\widetilde R_k);\ \TR_1 \in B_1, ..., \TR_{n-1} \in dy \biggr]
g(y) \P\{y + \widetilde{\xi}(y)\in B_n\}\\
&=& \E_x \biggl\{ \prod_{k=0}^{n-1} g(\widetilde R_k);\ \TR_1 \in B_1, ..., \TR_n \in B_n \biggr\},
\end{eqnarray*}
which completes the proof.
\end{proof}

Let $\widetilde{\psi}(x)$ denote the ruin probability for the chain $\{\TR_n\}$, that is,
$$
\widetilde{\psi}(x)=\P_x\{\TR_n<0\text{ for some }n\ge1\}.
$$
Let $\widetilde{H}_x$ be the renewal measure of $\{\TR_n\}$
with starting point $x$:
$$
\widetilde{H}_x(B)=
\sum_{n=0}^\infty\P_x\{\TR_n\in B\},
\quad B\in\mathcal{B}(\R).
$$

\begin{lemma}
The following inequality holds true
\begin{eqnarray}\label{eq:decomp}
\psi(x) &\le& \widetilde{\psi}(x)+\int_{0}^{\infty} (1-g(y)) \widetilde{H}_x (dy).
\end{eqnarray}
\end{lemma}

\begin{proof}
Let
$$
\tau_0:=\inf\{n\ge1: R_n< 0\}
$$
and 
$$
A_{\tau_0} := \{ \xi(R_k) \ge -R_k/2\text{ for all } k < \tau_0\}.
$$
Noting that 
$$
\{ \tau_0 < \infty \} \subseteq \left( \{ \tau_0 < \infty \} \cap A_{\tau_0} \right) \cup A_{\tau_0}^c,
$$ 
we get
\begin{align}\label{form:tau0up}
\P_x \{\tau_0 < \infty\} \le \P_x \{\tau_0 < \infty,A_{\tau_0}\} + \P_x \{A_{\tau_0}^c\}.
\end{align}
Using now Lemma \ref{lemma:mengenAn} with $B_1 = ... = B_{n-1} = [0,\infty)$ 
and $B_n = (-\infty, 0)$ we obtain
\begin{align*}
\P_x \{\tau_0 = n, A_n\} 
= \E_x \biggl\{\prod_{k=0}^{n-1} g(\widetilde R_k);\ \widetilde{\tau}_0 = n\biggr\} 
\le \P_x \{\widetilde{\tau}_0 = n\},
    \quad n\ge1,
\end{align*}
where
$$
\widetilde\tau_0:=\inf\{n\ge1:\ \TR_n< 0\}.
$$
This implies that 
\begin{align}
\label{A_tau}
\P_x \{\tau_0 < \infty, A_{\tau_0}\} 
\le \P_x \{\widetilde{\tau}_0 < \infty\}
=\widetilde{\psi}(x).
\end{align}
To bound the second probability term on the right hand side of (\ref{form:tau0up}),
we firstly apply the total probability law twice
\begin{eqnarray*}
\P_x \{A_{\tau_0}^c\} &=& \P_x\{\xi(R_k) <-R_k/2\text{ for some } k < \tau_0\}\\
&=& \sum_{n=0}^\infty \P_x \{A_n, \tau_0 > n, \xi(R_n)< -R_n/2\} \\
&=& \sum_{n=0}^{\infty} \int_0^\infty \P_x \{R_n \in dy, A_n, \tau_0 > n\}\P\{\xi(y)< -y/2\},
\end{eqnarray*}
and then we apply again Lemma \ref{lemma:mengenAn}
to the probability on the right hand side:
\begin{eqnarray*}
\P_x \{A_{\tau_0}^c\} &\le& 
\sum_{n=0}^{\infty} \int_0^\infty \P_x \{\widetilde R_n \in dy\} (1-g(y))\\
&=& \int_0^\infty (1-g(y)) \widetilde{H}_x (dy).
\end{eqnarray*}
Plugging now this bound and \eqref{A_tau} into (\ref{form:tau0up}), we get the desired 
upper bound.
\end{proof}
In order to get an upper bound for $\psi(x)$ we need upper bounds
for both terms on the right hand side of \eqref{eq:decomp}. 
It turns out that $\widetilde\psi(x)$ can be estimated by the method 
used in the proof of Theorem~\ref{thm:risk}.

\begin{lemma}\label{lem:psi-tilde}
Assume that $\E\tau^2\log(1+\tau)$ and $\E\xi^2\log(1+\xi)$ are finite.
Then there exists a constant $C$ such that
\begin{equation}\label{eq:down-trunc}
\P_x\{\TR_n\le y\text{ for some }n\ge1\}
\le C\frac{1+y^\rho}{x^\rho}\quad\mbox{for all }0<y<x.
\end{equation}
In particular,
$$
\widetilde{\psi}(x)\le\frac{C}{x^\rho}\quad\mbox{for all }x>0.
$$
\end{lemma}

\begin{proof}
Let $U_{-}(x)$ be the function defined in the proof of Theorem~\ref{thm:risk}. 
By the definition of $\widetilde{\xi}(x)$,
\begin{eqnarray*}
 \E U_-(x+\widetilde{\xi}(x))-U_-(x)
 &=&\frac{1}{g(x)}\E\{U_-(x+\xi(x))-U_-(x);\ \xi(x)\ge -x/2\}\\
 &=&\frac{1}{g(x)}\E\{U_-(x+\xi(x))-U_-(x);\ |\xi(x)|\le x/2\}\\
 && +\frac{1}{g(x)}\E\{U_-(x+\xi(x))-U_-(x);\ \xi(x)> x/2\}.\\
\end{eqnarray*}
Since the estimates \eqref{cond.3.moment.return} and  \eqref{risk.8.tail} are valid under 
the conditions of the present lemma, we may apply \eqref{L.harm2.2a.return} 
and \eqref{L.harm2.2.2.return} to get 
$$
\E U_-(x+\widetilde{\xi}(x))-U_-(x)
=-\frac{1+o(1)}{g(x)}p(x)e^{-Q_-(x)}.
$$
Therefore, there exists $\widehat{x}$ such that 
$U_-(\TR_{n\wedge\tau_B})$ is a bounded supermartingale,
where $B=(-\infty,\widehat{x}]$. 
Then, applying the optional stopping theorem, we get the desired upper bound.
\end{proof}

We now turn to the second term in \eqref{eq:decomp}. 
Firstly we state the following upper bounds for the renewal function.

\begin{lemma}\label{lem:H-tilde}
The following bounds hold true:
\begin{eqnarray}\label{eq:H1}
\widetilde{H}_x(0,y] &\le& C(1+y^2)\quad\mbox{for all }x,y>0
\end{eqnarray}
and 
\begin{eqnarray}\label{eq:H2}
\widetilde{H}_x(0,y] &\le& C\frac{1+y^{2+\rho}}{x^\rho}
\quad\mbox{for all }0<y<x.
\end{eqnarray}
\end{lemma}

\begin{proof}
Firstly note that
$$
\P\{\widetilde{\xi}(u)<-u/2\}=0.
$$
Next, by using Lemma~\ref{l:risk}, we conclude that 
$$
\E\widetilde{\xi}(u)\sim \frac{\theta\E\tau}{u}
$$
and 
$$
\E\widetilde{\xi}^2(u)\to \V\xi+v_c^2\V\tau\quad\mbox{as }u\to\infty.
$$
Using these estimates one can easily see that all conditions of Lemma 4 in 
\cite{DKW2013} are met. This implies \eqref{eq:H1}.
To prove \eqref{eq:H2} it suffices to notice that 
$$
\widetilde{H}_x(0,y]
\le \P_x\{\TR_n\le y\text{ for some }n\ge1\}
\sup_{u\le y}\widetilde{H}_u(0,y]
$$
and to apply \eqref{eq:down-trunc} and \eqref{eq:H1}.
\end{proof}

Now we are ready to bound the second term in \eqref{eq:decomp}. 
Since $\xi(x)\ge_{st} -\xi$,
\begin{eqnarray}\label{eq:1-g.xi}
\int_0^\infty (1-g(y))\widetilde{H}_x(dy) &\le&
\int_0^\infty \P\{\xi>y/2\}\widetilde{H}_x(dy).
\end{eqnarray}
Integration by parts implies that
\begin{eqnarray*}
\int_0^\infty \P\{\xi>y/2\}\widetilde{H}_x(dy) 
&=& \P\{\xi>y/2\}\widetilde{H}_x(0,y]\Big|_0^\infty
-\int_0^\infty \widetilde{H}_x(0,y]d\P\{\xi>y/2\}\\
&=& \int_0^\infty \widetilde{H}_x(0,y]\P\{2\xi\in dy\},
\end{eqnarray*}
because $\P\{\xi>y\}=o(1/y^2)$ as $y\to\infty$ due to $\E\xi^2<\infty$,
and $\widetilde{H}_x(0,y]=O(y^2)$ due to \eqref{eq:H1}.
Next, by Lemma \ref{lem:H-tilde},
\begin{eqnarray*}
\int_0^\infty \widetilde{H}_x(0,y]\P\{2\xi\in dy\} &\le&
C\int_0^x \frac{1+y^{2+\rho}}{x^\rho}\P\{2\xi\in dy\}
+C\int_x^\infty (1+y^2)\P\{2\xi\in dy\}\\
&\le& C_1 (1/x^\rho+x^2\P\{\xi\ge x/2\}),
\end{eqnarray*}
owing to the regular variation of the distribution of $\xi$
and Karamata's theorem.
Since $\beta<\rho$, we conclude that
\begin{eqnarray*}
\int_0^\infty \widetilde{H}_x(0,y]\P\{2\xi\in dy\} &\le&
C_2 x^2\P\{\xi>x\}.
\end{eqnarray*}
Together with \eqref{eq:1-g.xi} it yields that 
\begin{eqnarray*}
\int_0^\infty (1-g(y))\widetilde{H}_x(dy)
&\le& C_2 x^2\P\{\xi>x\}.
\end{eqnarray*}
This estimate and Lemma~\ref{lem:psi-tilde} imply the desired upper bound.

\subsection{Proof of the lower bound}

To star with, we notice that, for all $x>0$ and $N\ge 1$,
\begin{eqnarray*}
\psi(x) &=& \sum_{n=0}^\infty\int_0^\infty
\P_x\{R_n\in dy,\ R_k\ge 0\mbox{ for all }k\le n\} \P\{\xi(y)<-y\}\\
&\ge& \inf_{y\in[x/2,2x]}\P\{\xi(y)< -y\}
\sum_{n=0}^N \P_x\{R_k\in [x/2,2x]\mbox{ for all }k\le n\}\\
&\ge& \inf_{y\in[x/2,2x]}\P\{\xi(y)< -y\}
N\P_x\{R_k\in [x/2,2x]\mbox{ for all }k\le N\}.
\end{eqnarray*}
Due to Lemma~\ref{lem:g-tail}, 
\begin{eqnarray*}
\inf_{y\in[x/2,2x]}\P\{\xi(y)< -y\} &\ge& \inf_{y\in[x/2,2x]}\P\{\xi(y)< -2x\}\\
&\sim& \P\{\xi>2x\}\quad\mbox{as }x\to\infty.
\end{eqnarray*}
Consequently, putting $N=\delta x^2$, we get the following lower bound
\begin{eqnarray*}
\psi(x)&\ge& c\P\{\xi>x\} \delta x^2\P_x\{R_k\in [x/2,2x]\mbox{ for all }k\le \delta x^2\},
\end{eqnarray*}
for every $\delta>0$. Thus, it only remains to show that we can choose a $\delta>0$ 
so small that the probability on the right hand side is bounded away from zero.

We start by stating the following decomposition
\begin{eqnarray}\label{eq:lb1}
\nonumber
\lefteqn{\P_x\{R_k\not\in [x/2,2x]\mbox{ for some }k\le \delta x^2\}}\\
&\le&\P_x\{R_k\le x/2\text{ for some }k\ge1\}
+\P_x\Bigl\{\max_{k\le \delta x^2}R_k>2x,R_k\ge x/2\text{ for all }n\ge1\Bigr\}.\nonumber\\[-1mm]
\end{eqnarray}
It follows from Lemma~\ref{l:risk} that for every 
$\varepsilon<\rho$ there exists an $x_0$ such that 
$$
\frac{2xm_1(x)}{m_2(x)}\ge 1+\varepsilon
\quad \text{for all }x\ge x_0.
$$
Noting that 
$$
\P\{\xi(x)\le-\gamma x\} \le\P\{\xi>\gamma x\}=o(1/x^{2+\beta_0})
$$
for every $\beta_0<\beta$, we infer that all the conditions of Lemma~1 
in \cite{DKW2013} hold true and, consequently, there exists an $x_0$ such that
$$
\P_x\{R_n\le z\text{ for some }n\ge1\} \le (z/x)^{\beta_0}
\quad\text{for all }x>z>x_0.
$$
In particular,
\begin{equation}\label{eq:lb2}
\P_x\{R_k\le x/2\text{ for some }k\ge1\}\le 1/2^{\beta_0}
\quad\text{for all }x>2x_0.
\end{equation}
To bound the second probability on the right hand side of 
\eqref{eq:lb1} we introduce a martingale
$$
M_k:=R_k-R_0-\sum_{j=0}^{k-1}m_1(R_j),\quad k\ge0.
$$
Due to Lemma~\ref{l:risk}, we may assume that $x_0$ is so large that 
$ym_1(y)\le 2\theta\E\tau$ for all $y\ge x_0$. This implies that, for $R_0=x$,
\begin{eqnarray*}
\max_{k\le \delta x^2}R_k &\le& x+\max_{k\le \delta x^2}M_k+4\theta\delta x\E\tau 
\end{eqnarray*}
on the event $\{R_k\ge x/2\text{ for all }k\ge1\}$.
Consequently,
\begin{eqnarray*}
\P_x\Bigl\{\max_{k\le \delta x^2}R_k>2x,R_k\ge x/2\text{ for all }k\ge1\Bigr\}
&\le& \P_x\Bigl\{\max_{k\le \delta x^2}M_k>(1-c_1\delta)x\Bigr\},
\end{eqnarray*}
where $c_1:=4\theta\E\tau$. 
Applying the Doob inequality to the right hand side
and noting that $\E_x M_k^2\le c_2k$ for all $k$ and $x$, 
we obtain 
\begin{eqnarray*}
\P_x\Bigl\{\max_{k\le \delta x^2}R_k>2x,R_k\ge x/2\text{ for all }k\ge1\Bigr\}
&\le& \frac{c_2 \delta}{(1-c_1\delta)^2}.
\end{eqnarray*}
Plugging this estimate and \eqref{eq:lb2} into \eqref{eq:lb1}, we conclude that 
\begin{eqnarray*}
\P_x\{R_k\not\in [x/2,2x]\mbox{ for some }k\le \delta x^2\} &\le& 
\frac{1}{2^{\beta_0}}+\frac{c_2 \delta}{(1-c_1\delta)^2}
\quad\mbox{for all }x\ge 2x_0.
\end{eqnarray*}
Choosing $\delta>0$ sufficiently small, we can make the right hand side
less than $1$, hence 
\begin{eqnarray*}
\inf_{x\ge 2x_0}\P_x\{R_k\in [x/2,2x]\mbox{ for all }k\le \delta x^2\} &>& 0.
\end{eqnarray*}
This completes the proof of the lower bound.

\section{Appendix}
\label{sec:appendix}

\begin{lemma}\label{l:p.V.maj.o}
Let $\alpha\in(0,1]$ and $\gamma\ge\alpha$. 
Let a family of positive random variables $\{\xi_\theta,\ \theta\in\Theta\}$
possess a majorant $\Xi$ with $\gamma+1-\alpha$ moment finite, 
that is, $\E\Xi^{\gamma+1-\alpha}<\infty$ and
$$
\xi_\theta\ \le_{st}\ \Xi\quad\mbox{for all }\theta\in\Theta.
$$
Then there exists a decreasing integrable at infinity function $p(x)$ such that
$$ 
\sup_{\theta\in\Theta}\E\{\xi_\theta^{\gamma+1};\ \xi_\theta\le x\}
=o(x^{1+\alpha} p(x))\quad\text{as }x\to\infty.
$$
\end{lemma}

\begin{proof}
Integration by parts yields that
\begin{eqnarray*}
\E\{\xi_\theta^{\gamma+1};\ \xi_\theta\le x\}
&=& -\int_0^x y^{\gamma+1}d\P\{\xi_\theta>y\}\\
&=& -x^{\gamma+1}\P\{\xi_\theta>x\}
+(\gamma+1)\int_0^x y^\gamma\P\{\xi_\theta>y\}dy\\
&\le& (\gamma+1)\int_0^x y^\gamma\P\{\Xi>y\}dy,
\end{eqnarray*}
by the majorisation condition. Therefore, by the Markov inequality,
\begin{eqnarray*}
\E\{\xi_\theta^{\gamma+1};\ \xi_\theta\le x\}
&\le& (\gamma+1)\int_0^x y^\alpha\E\{\Xi^{\gamma-\alpha};\ \Xi>y\}dy\\
&=& (\gamma+1)x^{1+\alpha} p(x),
\end{eqnarray*}
where
$$
p(x)\ :=\ \frac{1}{x^{1+\alpha}}\int_0^x y^\alpha\E\{\Xi^{\gamma-\alpha};\ \Xi>y\}dy.
$$
The finiteness of $\E\Xi^{\gamma+1-\alpha}$ implies integrability at infinity of $p(x)$.
Indeed,
\begin{eqnarray*}
\int_0^\infty p(x)dx &=& \int_0^\infty \frac{dx}{x^{1+\alpha}}
\int_0^x y^\alpha \E\{\Xi^{\gamma-\alpha};\ \Xi>y\}dy\\
&=& \int_0^\infty y^\alpha\E\{\Xi^{\gamma-\alpha};\ \Xi>y\}dy 
\int_y^\infty\frac{dx}{x^{1+\alpha}}\\
&=& \frac{1}{\alpha}\int_0^\infty \E\{\Xi^{\gamma-\alpha};\ \Xi>y\}dy\\
&=& \frac{\E\Xi^{\gamma+1-\alpha}}{\alpha}\ <\ \infty,
\end{eqnarray*}
by the moment condition on $\Xi$.
In addition, the function $p(x)$ is decreasing because
\begin{eqnarray*}
\lefteqn{\frac{d}{dx}\frac{1}{x^{1+\alpha}}
\int_0^x y^\alpha\E\{\Xi^{\gamma-\alpha};\ \Xi>y\}dy} \\
&&\hspace{10mm}=\ -\frac{1+\alpha}{x^{2+\alpha}}
\int_0^x y^\alpha\E\{\Xi^{\gamma-\alpha};\ \Xi>y\}dy
+\frac{1}{x}\E\{\Xi^{\gamma-\alpha};\ \Xi>x\}\\
&&\hspace{10mm}\le\ -\frac{1+\alpha}{x^{2+\alpha}}
\E\{\Xi^{\gamma-\alpha};\ \Xi>x\}\int_0^x y^\alpha dy
+\frac{1}{x}\E\{\Xi^{\gamma-\alpha};\ \Xi>x\}\\
&&\hspace{10mm}=\ 0.
\end{eqnarray*}
The proof is complete due to the next Lemma \ref{l:g.fin.p}.
\end{proof}

\begin{lemma}\label{l:g.fin.p}
Let $p(x)>0$ be a decreasing function which is integrable at infinity.
Then there exists a decreasing integrable at infinity
function $p_1(x)>0$ such that $p_1(x)/p(x)\to\infty$ as $x\to\infty$.
\end{lemma}

\begin{proof}
Since $p(x)$ is integrable at infinity,
there exists an increasing sequence $n_k\to\infty$,
$k\ge 0$, such that $x_0=0$ and
$$
\int_{x_k}^\infty p(y)dy\ \le\ 1/k^2\quad\mbox{for all }k\ge 1.
$$
Define a continuous function $p_1(x)$ as follows:
\begin{eqnarray*}
p_1(x_k) &:=& (k+1)p(x_k),
\end{eqnarray*}
between $x_k$ and $x_{k+1}$ we define $p_1(x)$ piece-wise linearly.
Then the function $p_1(x)$ satisfies the condition
$p_1(x)/p(x)\to\infty$ as $x\to\infty$.
Since $p(x)$ decreases, the sequence $x_k$ may be chosen in such a way that
$$
(k+2)p(x_{k+1})\ <\ (k+1)p(x_k)\quad\mbox{for all }k\ge 1,
$$
which guarantees that the function $p_1(x)$ is decreasing.
In addition, its integral may be bounded as follows:
\begin{eqnarray*}
\int_0^\infty p(x)g(x)dx &\le& \sum_{k=0}^\infty(k+1)
\int_{n_k}^{n_{k+1}}p(x)dx\\
&=& \sum_{k=0}^\infty\int_{n_k}^\infty p(x)dx
\ \le\ \int_0^\infty p(x)dx+\sum_{k=1}^\infty 1/k^2\ <\ \infty,
\end{eqnarray*}
which completes the proof.
\end{proof}

\begin{lemma}[Denisov \cite{D2006}]\label{l:denis}
Let $p(x)>0$ be a decreasing function which is integrable at infinity.
Then there exists a decreasing integrable at infinity function $p_1(x)>0$
which dominates $p(x)$ and is regularly varying at infinity with index $-1$.
\end{lemma}

\begin{lemma}\label{l:maj.p.e.V}
Let $\xi\ge 0$ be a random variable and let $V(x)\ge 0$
be an increasing function such that $\E V(\xi)<\infty$.
Let $U(x)\ge 0$ be a function such that
the function $f(x):=V(x)/xU(x)$ increases and satisfies the condition
\begin{eqnarray}\label{con.on.f}
\sup_{x>1}\frac{f(2x)}{f(x)} &<& \infty.
\end{eqnarray}
Then there exists an increasing
function $s(x)\to\infty$ of order $o(x)$ such that
\begin{eqnarray*}
\E\{U(\xi);\ \xi>s(x)\} &=& o(p(x)xU(x)/V(x))
\quad\mbox{as }x\to\infty,
\end{eqnarray*}
where $p(x)$ is a decreasing integrable at infinity function
which is only determined by $\xi$ and $V(x)$.
\end{lemma}

\begin{proof}
Since $\E V(\xi)<\infty$, the decreasing function
$$
p_1(x)\ :=\ \E\{V(\xi)/\xi;\ \xi>x\}
$$
is integrable at infinity. Then by Lemmas \ref{l:g.fin.p} and \ref{l:denis},
$$
\E\{V(\xi)/\xi;\ \xi>x\}\ =\ o(p(x))\quad\mbox{as }x\to\infty,
$$
where a decreasing function $p(x)$ is integrable and regularly varying
at infinity with index $-1$. Hence, due to the increase of $V(x)/xU(x)$,
\begin{eqnarray*}
\E\{U(\xi);\ \xi>x\} &=&
\E\Bigl\{\frac{U(\xi)\xi}{V(\xi)} V(\xi)/\xi;\ \xi>x\Bigr\}\\
&\le& \frac{\E\{V(\xi)/\xi;\ \xi>x\}}{V(x)/xU(x)}\\
&=& o(p(x)xU(x)/V(x))\quad\mbox{as }x\to\infty.
\end{eqnarray*}
Therefore, for any $n\in\N$,
\begin{eqnarray*}
\E\{U(\xi);\ \xi>x/n\}\ =\ o(p(x)xU(x)/V(x))
\quad\mbox{as }x\to\infty
\end{eqnarray*}
because the function $p(x)$ is regularly varying at infinity
and owing to \eqref{con.on.f}.
This implies existence of level $s(x)=o(x)$
which delivers the stated result.
\end{proof}

\begin{lemma}\label{l:maj.p.e}
Let $\xi\ge 0$ be a random variable with finite $\gamma$th moment
for some $\gamma\in[1,\infty)$. Let $\alpha\in[1/\gamma,1]$.
Then there exists an increasing
function $s(x)\to\infty$ of order $o(x^\alpha)$ such that,
for all $\beta\in[0,\gamma-1/\alpha]$,
\begin{eqnarray*}
\E\{\xi^\beta;\ \xi>s(x)\} &=& o(p(x)/x^{\alpha(\gamma-\beta)-1})
\quad\mbox{as }x\to\infty,
\end{eqnarray*}
where $p(x)$ is a decreasing integrable at infinity function
which is only determined by $\xi$, $\gamma$, and $\alpha$.
\end{lemma}

\begin{proof}
Put $\eta=\xi^{1/\alpha}$ and $V(x)=x^{\alpha\gamma}$.
As follows from Lemma \ref{l:maj.p.e.V} with $U(x)=x^{\alpha\beta}$,
since $\E\xi^\gamma=\E V(\eta)<\infty$,
there exists a regularly varying at infinity
with index $-1$ function $p(x)$ which is integrable at infinity
and a function $s(x)=o(x)$ such that
\begin{eqnarray*}
\E\{\eta^{\alpha\beta};\ \eta>s(x)\} &=& o(p(x)xU(x)/V(x))\\
&=& o(p(x)/x^{\alpha(\gamma-\beta)-1})\quad\mbox{as }x\to\infty,
\end{eqnarray*}
which can be rewritten as
\begin{eqnarray*}
\E\{\xi^\beta;\ \xi>s^\alpha(x)\} &=& o(p(x)/x^{\alpha(\gamma-\beta)-1})
\quad\mbox{as }x\to\infty,
\end{eqnarray*}
and the proof is complete.
\end{proof}

\begin{lemma}\label{l:maj.p.e.log}
Let $\xi\ge 0$ be a random variable and let $V(x)$ be a non-negative function
such that $\E V(\xi)\log(1+\xi)<\infty$. Then there exists an increasing
function $s(x)\to\infty$ of order $o(x)$ such that,
\begin{eqnarray*}
\E\{V(\xi);\ \xi>s(x)\} &=& o(p(x)x)\quad\mbox{as }x\to\infty,
\end{eqnarray*}
where $p(x)$ is a decreasing integrable at infinity function.
\end{lemma}

\begin{proof}
It follows almost immediately because
\begin{eqnarray*}
\int_1^\infty \frac{\E\{V(\xi);\ \xi>x\}}{x}dx &=&
\int_1^\infty \frac{dx}{x}\int_x^\infty V(y)\P\{\xi\in dy\}\\
&=& \int_1^\infty V(y)\P\{\xi\in dy\}\int_1^y \frac{dx}{x}\\
&=& \int_1^\infty V(y)(\log y)\P\{\xi\in dy\}\ <\ \infty.
\end{eqnarray*}
Hence, by Lemmas \ref{l:g.fin.p} and \ref{l:denis},
\begin{eqnarray*}
\E\{V(\xi);\ \xi>x\} &=& o(p(x)x)\quad\mbox{as }x\to\infty,
\end{eqnarray*}
where a decreasing function $p(x)$ is integrable and regularly varying
at infinity with index $-1$. Then concluding arguments as
in Lemma \ref{l:maj.p.e.V} complete the proof.
\end{proof}

\begin{lemma}\label{l:g.fin.p.2}
Let $p(x)>0$ be a decreasing function which is integrable at infinity.
Then, for any $k\ge 1$, there exists a decreasing integrable at infinity
function $p_k(x)\ge p(x)$ such that it is $k$ times
differentiable and, for all $j\le k$,
$$
\frac{d^j}{dx^j}p_k(x)\ =\ O(1/x^{1+j})\quad\mbox{as }x\to\infty.
$$
\end{lemma}

\begin{proof}
Consider a decreasing function $p_k(x)$ defined by the equality
\begin{eqnarray*}
p_k(x) &:=& 2^k\int_{x/2}^\infty dy_k\int_{y_k/2}^\infty dy_{k-1}
\ldots\int_{y_3/2}^\infty dy_2
\int_{y_2/2}^\infty\frac{p(y_1)}{y_1^k}dy_1.
\end{eqnarray*}
Firstly, since the function $p(x)/x^k$ decreases,
\begin{eqnarray*}
\int_{y_2/2}^\infty\frac{p(y_1)}{y_1^k}dy_1 &\ge&
\int_{y_2/2}^{y_2}\frac{p(y_1)}{y_1^k}dy_1
\ \ge\ \frac{y_2}{2}\frac{p(y_2)}{y_2^k}
\ =\ \frac{1}{2}\frac{p(y_2)}{y_2^{k-1}},
\end{eqnarray*}
so repetition of this lower bound eventually
leads to the inequalities
\begin{eqnarray*}
p_k(x) &\ge& 2^k\int_{x/2}^x \frac{1}{2^{k-1}}\frac{p(y_k)}{y_k}dy_k
\ \ge\ 2^k\frac{x}{2}\frac{1}{2^{k-1}}\frac{p(x)}{x}\ =\ p(x).
\end{eqnarray*}

Secondly, $p_k(x)$ is integrable at infinity because
\begin{eqnarray*}
\int_{y_2/2}^\infty\frac{p(y_1)}{y_1^k}dy_1 &\le&
p(y_2/2)\int_{y_2/2}^\infty\frac{1}{y_1^k}dy_1
\ =\ O\Bigl(\frac{p(y_2/2)}{y_2^{k-1}}\Bigr),
\end{eqnarray*}
and hence after $k-1$ steps we arrive at upper bound
\begin{eqnarray*}
p_k(x) &\le& c\int_{x/2}^\infty\frac{p(y_k/2^{k-1})}{y_k}dy_k,
\quad c<\infty,
\end{eqnarray*}
where the integral on the right hand side is integrable with respect
to $x$, since
\begin{eqnarray*}
\int_0^\infty dx\int_{x/2}^\infty\frac{p(y/2^{k-1})}{y}dy
&=& \int_0^\infty \frac{p(y/2^{k-1})}{y}dy\int_0^{2y} dx\\
&=& 2\int_0^\infty p(y/2^{k-1})dy\ <\ \infty.
\end{eqnarray*}

Thirdly,
\begin{eqnarray*}
\frac{d^k}{dx^k}p_k(x) &=&
-\frac{2^k}{2} \frac{d^{k-1}}{dx^{k-1}} \int_{x/4}^\infty dy_{k-1}
\ldots\int_{y_3/2}^\infty dy_2
\int_{y_2/2}^\infty\frac{p(y_1)}{y_1^k}dy_1\\
&\ldots& \\
&=& (-1)^k\frac{2^k}{2\cdot 4\cdot\ldots\cdot 2^k}
\frac{p(x/2^k)}{(x/2^k)^k}\ =\ O(p(x/2^k)/x^k)\quad\mbox{as }x\to\infty.
\end{eqnarray*}
Since $p(x)$ is decreasing and integrable at infinity,
$p(x)=O(1/x)$ as $x\to\infty$, so $p_k^{(k)}(x)=O(1/x^{1+k})$.
Integrating the $k$th derivative $k-j$ times we get that the $j$th
derivative of $p_k(x)$ is not greater than $(k-j)$th
integral of $c/x^{1+k}$ which is of order $O(1/x^{1+j})$.
This completes the proof.
\end{proof}

\end{document}